\documentclass[10pt,reqno]{amsart}
\usepackage{amsmath}
\usepackage{amssymb}
\usepackage{amsthm}

\newlength{\defbaselineskip}
\newcommand{\setlinespacing}[1]%
           {\setlength{\baselineskip}{#1 \defbaselineskip}}

\numberwithin{equation}{section}

\newtheorem{thm}{Theorem}[section]

\newtheorem{lem}[thm]{Lemma}
\newtheorem{prop}[thm]{Proposition}

\theoremstyle{definition}
\newtheorem{defn}[thm]{Definition}

\theoremstyle{remark}
\newtheorem{rem}[thm]{Remark}
\numberwithin{equation}{section}

\begin{document}
\title[Global attractor for the Kawahara equation]
{Global attractor for the weakly damped forced Kawahara equation on the torus}

\author{Jaeseop Ahn, Seongyeon Kim and Ihyeok Seo}

\thanks{This research was supported by the Research Grant of Jeonju University in 2024 (S. Kim), and by NRF-2022R1A2C1011312 (I. Seo).}


\address{Department of Mathematics, Sungkyunkwan University, Suwon 16419, Republic of Korea}
\email{j.ahn@skku.edu}

\address{Department of Mathematics Education, Jeonju University, Jeonju 55069, Republic of Korea}
\email{sy\_kim@jj.ac.kr}

\address{Department of Mathematics, Sungkyunkwan University, Suwon 16419, Republic of Korea}
\email{ihseo@skku.edu}

\begin{abstract}
We study the long time behaviour of solutions for the weakly damped forced Kawahara equation on the torus.
More precisely, we prove the existence of a global attractor in $L^2$, to which as time passes all solutions draw closer.
In fact, we show that the global attractor turns out to lie in a smoother space $H^2$ and be bounded therein. 
Further, we give an upper bound of the size of the attractor in $H^2$ that depends only on the damping parameter and the norm of the forcing term.
\end{abstract}

\maketitle

\section{Introduction}\label{sec1}
In this paper we consider the weakly damped forced Kawahara equation on the torus:
\begin{equation}\label{fwdk}
\begin{aligned}
&\partial_t u +\alpha \partial_x^5 u+\beta \partial_x^3 u+\gamma u+ \frac12\partial_x(u^2)=f,\ \ x\in\mathbb T=\mathbb R/(2\pi\mathbb{Z}),\ \,t\ge0,\\
&u(x,0)=g(x)\in\dot{L}^2(\mathbb T):=\{h\in L^2(\mathbb T):\int_\mathbb Th(x)dx=0\},
\end{aligned}
\end{equation}
where $\alpha,\beta\in\mathbb R$ with $\alpha\ne0$. We assume $u$ and $f$ are real-valued, and the forcing term $f\in\dot{L}^2(\mathbb T)$ is independent of $t$. Here, \eqref{fwdk} is dissipative due to the damping parameter $\gamma>0$ (c.f. \eqref{absorb}) as opposed to its Hamiltonian counterpart with $\gamma=0$.
It should be noted that the mean-zero property is conserved over time since
$$\partial_t\int_\mathbb Tu(x,t)dx=-\gamma\int_\mathbb Tu(x,t)dx\quad\text{and}\quad\int_\mathbb Tu(x,0)dx=0.$$
This fifth-order KdV type equation has been derived to model the capillary waves on a shallow layer and the magneto-sound propagation in plasma (see e.g. \cite{Kaw,KS}).

Global attractors have been widely studied as a basic concept to describe the long time asymptotics of solutions to nonlinear PDEs with dissipation.
Given a globally well-posed problem with the associated semigroup operator $U(t):g\in H\rightarrow u(t)\in H$ with the phase space $H$, the study is focused on a compact invariant set $X\subset H$ (a global attractor) to which the orbit $u(t)$ converges as $t\rightarrow\infty$:
$$U(t)X=X,\quad t\ge0,\qquad\text{dist}(u(t),X)\rightarrow0.$$
We introduce some relevant definitions.

\begin{defn}
An attractor is a set $\mathcal A\subset H$ that is invariant under the flow $U(t)$ and possesses an open neigbourhood $\mathcal U$ such that, for every $g\in\mathcal U$, the distance $\text{dist}(U(t)g,\mathcal A)$ in $H$ goes to zero as $t \rightarrow \infty$.
We say that $\mathcal A$ attracts the points of $\mathcal U$ and we call the largest open such set $\mathcal U$ the basin of attraction of $\mathcal A$.
\end{defn}

\begin{defn}
If $\mathcal A\subset H$ is a compact attractor whose basin of attraction is the whole phase space $H$,
then $\mathcal A$ is called a global attractor for the semigroup $\{U(t)\}_{t\ge0}$.
\end{defn}

We postpone the issue of the global well-posedness of \eqref{fwdk} to Section \ref{sec2}, and first state the following result on the existence of the global attractor. 
 
\begin{thm}\label{ga}
Consider the weakly damped forced Kawahara equation \eqref{fwdk} with $\frac{3\beta}{5\alpha}\not\in\mathbb Z^+$. Then the equation possesses a global attractor in $\dot L^2$. Moreover, for any $s\in(0,2)$, it is a compact subset of $H^s$, and is bounded in $H^2$ by a constant depending only on $\alpha$, $\beta$, $\gamma$ and $\|f\|$ where $\|\cdot\|:=\|\cdot\|_{L^2(\mathbb T)}$.
\end{thm}

Precisely the condition we need is $\frac{3\beta}{5\alpha}\ne k_1^2+k_2^2+k_1k_2$ for any $k_1,k_2\in\mathbb Z\setminus \{0\}$ which is weaker than $\frac{3\beta}{5\alpha}\not\in\mathbb Z^+$; without such, we encounter resonant terms that render it difficult to implement the arguments in this paper. Thus the condition shall be assumed throughout this paper and omitted henceforward. For the weakly damped forced KdV equation on the torus, the existence of a global attractor was shown by many authors; see \cite{ET2} and the references therein. 
However, such is unknown for the Kawahara equation \eqref{fwdk} yet, which is a fifth-order KdV type equation (See \cite{C} for a related result).

\

\noindent\textit{Notation.}\label{sec1.2}
Throughout this paper, the symbol $C$ stands for a positive constant which may be different at each occurrence.
We write $A\lesssim B$ to mean $A\le CB$ with an unspecified constant $C>0$.
We also write $A\sim B$ to denote both $A\lesssim B$ and $B\lesssim A$. We define $\langle\cdot\rangle:=1+|\cdot|$.
The Fourier transform and Fourier sequence of $u$ on $\mathbb R$ and $\mathbb T$ are given by
$$\widehat u(\tau)=\frac1{2\pi}\int_\mathbb R e^{-i\tau t} u(t)dt$$
and
$$u_k:=\widehat u(k)=\frac1{2\pi}\int_0^{2\pi}e^{-ikx}u(x)dx,\quad k\in\mathbb Z,$$
respectively.
For multivariable functions, $\widehat u$ denotes the Fourier transform of $u$ in all variables. We shall use the notations $\widehat u$ and $\mathcal Fu$ interchangeably, and for one variable transforms of multivariable functions we specify the variable. For example, $\mathcal F_xu$ denotes the space Fourier transform of $u$.
Note finally that for an $\dot L^2$ function $u$, $\|u\|_{H^s}\sim\|u_k|k|^s\|_{l^2}$. 

\

\noindent\textit{Outline of the paper.}
In Section \ref{sec2} we introduce function spaces and discuss energy bound. 
Although the conservation of energy does not hold for the damped equation \eqref{fwdk},
the energy remains bounded for positive times. 
Then we establish local well-posedness and draw global well-posedness from the energy bound. 
In Section \ref{sec3} we prove Theorem \ref{ga} by looking at linear and non-linear evolutions of the solution separately and investigating their asymptotic behaviours; the non-linear evolution lies in a smoother space by non-linear smoothing while the linear evolution vanishes.
We obtain the smoothing property in Section \ref{sec4} using some arguments similar to \cite{ET2}.
The final section, Section \ref{sec6}, is devoted to the proof of Lemmas \ref{lem2} and \ref{lem}, which are needed to set up the well-posedness in Section \ref{sec2}.

\section{Local and global well-posedness}\label{sec2}

For the periodic Kawahara equation ($\gamma=f=0$ in \eqref{fwdk}),
Hirayama \cite{H} proved that the solution is locally well-posed in $\dot{H}^s(\mathbb T)$ for $s\ge-1$ by adapting the argument used in  \cite{KPV} for the KdV equation. 
In \cite{Kat2}, Kato improved the result to $s\ge-3/2$, and proved $C^3$-ill-posedness for $s<-3/2$ as well as global well-posedness for $s\ge-1$.

Here we establish local and global well-posedness of the weakly damped forced case \eqref{fwdk}. First we present appropriate function spaces. 
Let $X^{s,b}$ denote the Bourgain space under the norm
$$\|u\|_{X^{s,b}}:=\|\langle k\rangle^s\langle\tau+\alpha k^5-\beta k^3\rangle^b\widehat u(k,\tau)\|_{l_k^2L_\tau^2}.$$
We also define, for $\delta\ge0$, the restricted space $X_{\delta}^{s,b}$ by the norm
$$\|u\|_{X_\delta^{s,b}}:=\inf_{\tilde u=u\ \text{on}\ t\in[0,\delta]}\|\widetilde u\|_{X^{s,b}}.$$
We further introduce the $Y^s$- and $Z^s$-spaces based on the ideas of Bourgain \cite{B} (cf. \cite{GTV,CKSTT1,ET2}), which are defined via the norms
\begin{align*}
	&\|u\|_{Y^s}:=\|u\|_{X^{s,\frac12}}+\|\langle k\rangle^s\widehat u(k,\tau)\|_{l_k^2L_\tau^1},\\
	&\|u\|_{Z^s}:=\|u\|_{X^{s,-\frac12}}+\bigg\|\frac{\langle k\rangle^s\widehat u(k,\tau)}{\langle\tau+\alpha k^5-\beta k^3\rangle}\bigg\|_{l_k^2L_\tau^1}.
\end{align*}
One defines $Y_\delta^s$, $Z_\delta^s$ accordingly. We note here that if $u\in Y^s$ then $u\in L_t^\infty H_x^s$ (In fact, $Y^s$ embeds into $C_tH_x^s$); the $X^{s,\frac12}$-norm alone fails to control the $L_t^\infty H_x^s$-norm of the solution.

The following theorem shows the well-posedness of \eqref{fwdk} in $L^2$. 

\begin{thm}\label{thm2}
The Cauchy problem \eqref{fwdk} is locally well-posed in $L^2$. Namely, there exist $\delta=\delta(\|f\|,\|g\|,\gamma)$ and a unique solution $u\in C([0,\delta];L_x^2(\mathbb T))\cap Y_\delta^0$ with
	\begin{equation}\label{bd}
	\|u\|_{X_\delta^{0,\frac12}}\le\|u\|_{Y_\delta^0}\lesssim\|g\|.
	\end{equation}
Furthermore, the local solution extends globally in time.
\end{thm}
Before proceeding, we look at dissipativity of \eqref{fwdk}. Whilst we do not have energy conservation, the energy remains bounded for positive times and in the long run it is less than $2\|f\|/\gamma$.

To obtain the energy bound, note that
\begin{align*}
	\partial_t\int_\mathbb Tu^2dx&=2\int_\mathbb Tu\partial_tudx\\
	&=-2\alpha\int_\mathbb Tu\partial_x^5udx-2\beta\int_\mathbb Tu\partial_x^3udx-\int_\mathbb Tu\partial_xu^2dx-2\gamma\int_\mathbb Tu^2dx+2\int_\mathbb Tfudx\\
	&=-2\gamma\int_\mathbb Tu^2dx+2\int_\mathbb Tfudx.
\end{align*}
Multiply by $e^{2\gamma t}$ and we get
$$\partial_t(e^{2\gamma t}\|u\|^2)=2e^{2\gamma t}\int_\mathbb Tfudx\le2e^{2\gamma t}\|f\|\|u\|$$
and
$$\frac{\partial_t(e^{2\gamma t}\|u\|^2)}{2e^{\gamma t}\|u\|}=\partial_t(e^{\gamma t}\|u\|)\le e^{\gamma t}\|f\|.$$
Therefore, for positive times
\begin{equation}\label{absorb}
	\|u(t)\|\le e^{-\gamma t}\|g\|+\frac{\|f\|}\gamma(1-e^{-\gamma t}).
\end{equation}
Thus for $t>T=T(\|f\|,\|g\|,\gamma)$, we have 
\begin{equation}\label{ab}
\|u(t)\|\le2\|f\|/\gamma.
\end{equation}
\begin{proof}[Proof of Theorem \ref{thm2}]
By Duhamel's principle, we wish to have a unique fixed point of the mapping $\Phi:Y_\delta^0\to Y_\delta^0$ given by
\begin{equation} \label{sol}
	\Phi(u)=e^{Lt}g-\int_0^te^{L(t-t')}F(u)(\cdot,t')dt',
\end{equation}
where $F(u)=\frac12\partial_x(u^2)+\gamma u-f$ and $L=-\alpha\partial_x^5-\beta\partial_x^3$.
We shall show that $\Phi$ defines a contraction map on 
$$\mathcal S:=\left\{u\in Y_{\delta}^0:\|u\|_{Y_{\delta}^0}\le M\right\}.$$
for large $M$. We first need the lemmas below.
\begin{lem}\label{lem2}
Let $s\in\mathbb R$. For any $\delta<1$, the following estimates hold:
	\begin{align}
		\label{w1}&\|e^{Lt}g\|_{Y_\delta^s}\lesssim\|g\|_{H^s},\\
		\label{w2}&\bigg\|\int_0^te^{L(t-t')}F(\cdot,t')dt'\bigg\|_{Y_\delta^s}\lesssim\|F\|_{Z_\delta^s}.
	\end{align}
\end{lem}
\begin{lem}\label{lem}
Let $s\ge0$ and $-1/2<b<b'<1/2$. 
For any $\delta<1$, the following estimates hold:
\begin{align}
\label{w4}&\|u\|_{X_\delta^{s,b}}\lesssim_{b,b'}\delta^{b'-b}\|u\|_{X_\delta^{s,b'}},\\
\label{w5}&\|u\|_{Z_\delta^s} \lesssim\delta^{1-}\|u\|_{Y_\delta^s},\\
\label{w3}&\|\partial_x(u_1u_2)\|_{Z_\delta^s}\lesssim_{\alpha,\beta,s}\|u_1\|_{X_\delta^{s,\frac3{10}}}\|u_2\|_{X_\delta^{s,\frac3{10}}}
	\end{align}
where $u_1,u_2$ are mean-zero functions. 
\end{lem}

\begin{rem}
Applying the forcing $f$ to \eqref{w5}, we get a stronger estimate:
	\begin{equation}\label{157}
	\|f\|_{Z_\delta^s} \lesssim\delta^{1-}\|f\|_{H^s}.
	\end{equation}
See the first paragraph in Subsection \ref{subsec5.2}.
\end{rem}

Assuming these lemmas for the moment, which will be proved in Section \ref{sec6}, we first show that $\Phi(u)\in Y_\delta^0$ for $u\in Y_\delta^0$.
For this, we apply \eqref{w1} and \eqref{w2} to the linear term and the Duhamel term in \eqref{sol} respectively to get
$$\|e^{Lt}g\|_{Y_\delta^0} \le C\|g\|,$$
\begin{equation}\label{duh}
\left\| \int_0^te^{L(t-t')}F(u)(\cdot,t')dt'\right\|_{Y_\delta^0}
\le C( \|\partial_x(u^2)\|_{Z_\delta^0}+\gamma\|u\|_{Z_\delta^0}+\|f\|_{Z_\delta^0}).
\end{equation}
Using Lemma \ref{lem} and \eqref{157}, we bound the right-hand side of \eqref{duh} as 
\begin{align*}
C(\|\partial_x(u^2)\|_{Z_\delta^0}+\gamma\|u\|_{Z_\delta^0}+\|f\|_{Z_\delta^0})
&\le C (\|u\|_{{X^{0,\frac3{10}}_\delta}}^2+\gamma\delta^{1-}\|u\|_{Y_\delta^0}+\delta^{1-}\|f\|)\\
&\le C( \delta^{\frac{2}{5}-}\|u\|_{{X_\delta^{0,\frac{1}{2}-}}}^2+\gamma\delta^{1-}\|u\|_{Y_\delta^0}+
\delta^{1-}\|f\|).
\end{align*}
Since $\|u\|_{X_\delta^{0,\frac{1}{2}-}}\le\|u\|_{X_\delta^{0,\frac12}}\le\|u\|_{Y_\delta^{0}}$,
we therefore get 
$$\|\Phi(u)\|_{Y_\delta^0} \le C\|g\|+  C (\delta^{\frac{2}{5}-} M^2+\gamma\delta^{1-}M+\delta^{1-}\|f\|).$$
For a fixed $M\ge2C(\|g\|+\|f\|/\gamma)$, if we take $\delta$ sufficiently small so that 
\begin{equation}\label{qs}
C(\delta^{\frac{2}{5}-}M+\gamma\delta^{1-})\le \frac{1}{2},
\end{equation}
then we arrive at $\|\Phi(u)\|_{Y_\delta^0}\le M.$

Next we show that $\Phi$ is a contraction on $\mathcal S$. Again using Lemma \ref{lem}, we see that
\begin{align*}
\|\Phi (u)-\Phi (v)\|_{Y_\delta^0}&\le C ( \|\partial_x(u^2-v^2)\|_{Z_\delta^0}+\gamma\left\|u-v\right\|_{Z_\delta^0} )\\
&\le C(\delta^{\frac{2}{5}-}\|u+v\|_{Y_\delta^{0}}+\gamma\delta^{1-})\|u-v\|_{Y_\delta^{0}}\\
&\le C (\delta^{\frac{2}{5}-}M+\gamma\delta^{1-})\|u-v\|_{Y_\delta^0}.
\end{align*}
With the same choice of $\delta$ as before, by \eqref{qs}, we confirm that $\Phi$ is a contraction on $\mathcal{S}$. 
Now the contraction mapping theorem asserts that there exists a unique solution $u\in\mathcal S$.

Finally the global well-posedness follows from the fact that the energy $\|u(t)\|$, by \eqref{absorb}, is bounded by $\|g\|+\|f\|/\gamma$ at all times $t\ge0$.
\end{proof}

\section{Global attractor}\label{sec3}
In this section we prove Theorem \ref{ga} with a known criterion for the existence of a global attractor, Theorem \ref{thmA}. For that we begin with a definition.

\begin{defn}
	Let $\mathcal B_0$ be a bounded set of $H$ and $\mathcal U$ an open set containing $\mathcal B_0$. We say that $\mathcal B_0$ is absorbing in $\mathcal U$ if for any bounded $\mathcal B\subset\mathcal U$ there exists $T = T(\mathcal B)$ such that $U(t)\mathcal B\subset\mathcal B_0$ for all $t\ge T$. We also say that $\mathcal B_0$ absorbs the bounded subsets of $\mathcal U$.
\end{defn}

Now we state how to construct a global attractor from an absorbing set.

\begin{thm}[\cite{Te}]\label{thmA}
	Let $H$ be a metric space and $ U(t):H\rightarrow H$ be a continuous semigroup for all $t\ge0$. Assume that there exist an open set $\mathcal U$ and a bounded set $\mathcal B_0$ of $\mathcal U$ such that $\mathcal B_0$ is absorbing in $\mathcal U$. If $\{U(t)\}_{t\ge0}$ is asymptotically compact, i.e. for every bounded sequence $x_j$ in $H$ and every sequence $t_j\rightarrow\infty$, $\{U(t_j)x_j\}_j$ is precompact in $H$, then 
	$$\omega(\mathcal B_0):=\bigcap_{s\ge0}\overline{\bigcup_{t\ge s}U(t){\mathcal B_0}}$$ is a compact attractor with basin $\mathcal U$ and it is maximal for the inclusion relation. (Here the closure is taken on $H$.)
\end{thm}

The following proposition, which will be proved in the next section, is enough to check the hypothesis of the above theorem. 
The proposition says that every $\dot L^2$ solution has its nonlinear part “trapped” in a ball in $H^2$ centred at zero of radius depending only on $\|f\|$, $\|g\|$ and $\gamma$.
One can explicitly calculate an upper bound for this radius by keeping track of the constants in our proof.

\begin{prop}\label{thm1}
		Let $u$ be the solution to \eqref{fwdk}. 
		Then for $L=-\alpha\partial_x^5-\beta\partial_x^3$, $$u(t)-e^{-\gamma t}e^{Lt}g\in C_t^0H_x^2.$$ 
		Furthermore,
		\begin{equation}\label{j}
		\|u(t)-e^{-\gamma t}e^{Lt}g\|_{H^2}\le C(\|f\|,\|g\|,\gamma).
		\end{equation}
	\end{prop}

Now we are ready to prove Theorem \ref{ga}.
The existence of an absorbing set follows from \eqref{ab} immediately.
To apply Theorem \ref{thmA}, we next show the asymptotic compactness of the propagator $U(t)$; for every bounded sequence $g_j \in L^2$ and every sequence $t_j \rightarrow \infty$, we must show that $\{U(t_j)g_{j}\}_j$ is precompact in $L^2$.
By decomposing the solution into linear and nonlinear parts, and by Proposition \ref{thm1}, 
we first see that 
$$U(t_j)g_j=e^{(-\gamma+L)t_j}g_j+N(t_j)g_j$$ 
where $N(t_j)g_j$ is in a ball in $H^2$ with a radius depending on $\|f\|$, $\|g\|$ and $\gamma$.
By Rellich's theorem\footnote{The inclusion $H^{s'}(\Omega)\hookrightarrow H^s(\Omega)$ is compact for $s<s'$ where $\Omega$ is a bounded domain in $\mathbb{R}^d$. Namely, any bounded set in $H^{s'}$ is precompact in $H^s$.}, $\{N(t_j)g_{j}\}_j$ is then precompact in $L^2$.
For the linear part, note that
$$\|e^{(-\gamma+L)t_j}g_j\|=e^{-\gamma t_j}\|g_j\| \lesssim e^{-\gamma t_j}$$
uniformly in $j$ since $\{g_j\}_j$ is assumed to be bounded in $L^2$.
Thus the linear part diminishes to zero in $L^2$ as $t_j\to\infty$.
Therefore, we conclude that both the linear and nonlinear parts are asymptotically compact in $L^2$, which yields that $\omega(\mathcal B_0)$ is a global attractor in $L^2$ as desired.

Finally, we show that $\omega(\mathcal B_0)$ is a compact subset of $H^s$ for any $s\in(0,2)$. 
For this, we note that 
\begin{equation}\label{co}
\|u(t)-e^{-\gamma(t-T)}e^{L(t-T)}u(T)\|_{H^2}\le C(\|f\|, \gamma)
\end{equation}
for $t\ge T= T(\|f\|,\|g\|,\gamma)$.
This easily follows from \eqref{j} and \eqref{ab}.
Let $B$ be the closed ball centred at zero in $H^2$ of radius $C(\gamma,\|f\|)$ given in \eqref{co}.
Then $B$ is compact in $H^s$, $s\in(0,2)$ by Rellich's theorem.
Since $$\omega(\mathcal B_0)=\bigcap_{\tau\ge0}\overline{\bigcup_{t\ge\tau}U(t)\mathcal B_0}$$
is obviously closed and any closed subset of a compact set is naturally compact, all we have to do is to show $\omega(\mathcal{B}_0) \subset B$.
By \eqref{co}, 
$$U_\tau:=\overline{\bigcup_{t\ge\tau}U(t)\mathcal B_0}\quad \text{for}\quad \tau>T $$
is contained in a $\delta_\tau$-neighbourhood $N_\tau$ of $B$ in $L^2$, where $\delta_\tau\rightarrow0$ as $\tau\rightarrow\infty$. Therefore, 
$$\omega(\mathcal B_0)=\bigcap_{\tau\ge0}U_\tau\subset\bigcap_{\tau\ge0}N_\tau=B.$$

\section{Smoothing property: proof of Proposition \ref{thm1}.}\label{sec4}
Now we prove Proposition \ref{thm1}.
We first write the equation \eqref{fwdk} on the Fourier transform side as in \eqref{fwdk3} and substitute $u_k(t)=w_k(t)e^{-\gamma t-i(\alpha k^5-\beta k^3)t}$.
We then use differentiation by parts to drag the phase down from the oscillation arising in this process. 
This makes it possible to take advantage of decay from the resulting large denominators. 
On the other hand, the order of the nonlinearity is increased from quadratic to cubic,
but we decompose the related trilinear term into two parts, resonant and nonresonant terms, and use Bourgain's restricted norm method to control the delicate nonresonant ones.

\subsection{Representation of Fourier coefficients of the solution}\label{sub4-1}
The following lemma gives an expression of the solution to \eqref{fwdk} on the Fourier transform side.
\begin{lem}
The solution $u$ to \eqref{fwdk} is represented in the following form:
	\begin{align}\label{morphback}
\nonumber
u_k(t)&=e^{-\gamma t-i(\alpha k^5-\beta k^3)t}g_k+\mathcal B(u,u)_k(t)-e^{-\gamma t-i(\alpha k^5-\beta k^3)t}\mathcal B(u,u)_k(0)\\
\nonumber
&\quad+\int_0^te^{-\gamma(t-r)-i(\alpha k^5-\beta k^3)(t-r)}(\gamma\mathcal B(u,u)_k(r)-2\mathcal B(u,e^{\gamma t - Lt}f)_k(r))dr\\
\nonumber
&\quad+\int_0^te^{-\gamma(t-r)-i(\alpha k^5-\beta k^3)(t-r)}(f_k + \rho(u)_k + \sigma(u)_k)
dr\\
&\quad+\int_0^te^{-\gamma(t-r)-i(\alpha k^5-\beta k^3)(t-r)}\mathcal R(u)_k(r)dr,
\end{align}
where for $k\ne0$
\begin{align*}
&\mathcal B(\phi,\psi)_k=-\frac12\sum_{k_1+k_2=k}\frac{\phi_{k_1}\psi_{k_2}}{k_1k_2\{5\alpha(k_1^2 +k_2^2 + k_1k_2)-3\beta\}},\\ &\rho(u)_k=\frac{i|u_k|^2u_k}{2(15\alpha k^2-3\beta)k}, \quad \sigma(u)_k=-iu_k\sum_{|j|\ne|k|}\frac{|u_j|^2}{j\{5\alpha(k^2-kj+j^2)-3\beta\}},\\
&\mathcal{R}(u)_k=-\frac i2\sum_{k_1+k_2+k_3=k\atop(k_1+k_2)(k_2+k_3)(k_3+k_1)\ne0}\frac{u_{k_1} u_{k_2} u_{k_3}} {k_1\{5\alpha(k_1^2 +(k_2+k_3)^2 + k_1(k_2+k_3))-3\beta\}},
\end{align*}
and $B(\phi,\psi)_0=\rho(u)_0=\sigma(u)_0=R(u)_0=0$. 
\end{lem}
\begin{proof}
Applying the Fourier transform to the equation \eqref{fwdk} implies
\begin{equation}\label{fwdk3}
	\begin{cases}
		\partial_tu_k=-i\alpha k^5u_k+ i\beta k^3u_k-\gamma u_k -\frac{ik}2\sum_{k_1+k_2=k}u_{k_1}u_{k_2}+f_k, \\
		u_k(0) = g_k.
\end{cases}
\end{equation}
Since $u$ and $f$ satisfy the mean-zero assumption, there are no zero harmonics in this equation. Using the transformation
$u_k(t)=:w_k(t)e^{-\gamma t-i(\alpha k^5-\beta k^3)t}$
and the identities
\begin{align*}
	(k_1+k_2)^3-(k_1^3+k_2^3)&=3k_1k_2(k_1+k_2),\\
	(k_1+k_2)^5-(k_1^5+k_2^5)&=5k_1k_2(k_1+k_2)(k_1^2+k_2^2+k_1k_2),
\end{align*}
the equation \eqref{fwdk3} can be written in the form
\begin{equation}\label{morph}
	\partial_tw_k=-\frac{ik}2e^{-\gamma t}\sum_{k_1+k_2=k}e^{i\{5\alpha(k_1^2 + k_2^2 + k_1 k_2 )-3\beta\} k_1k_2kt}w_{k_1}w_{k_2} + e^{\gamma t+i(\alpha k^5-\beta k^3)t} f_k.
\end{equation}

Using differentiation by parts with $e^{at}= \partial_t (\frac1a e^{at})$, we then rewrite \eqref{morph} as
\begin{align}\label{morph1}
\nonumber
\partial_tw_k
	&=-\partial_t\bigg(e^{-\gamma t}\sum_{k_1+k_2=k}\frac{e^{ik_1k_2k\{5\alpha(k_1^2 + k_2^2 + k_1 k_2 )-3\beta\}t}w_{k_1}w_{k_2}}{2k_1k_2\{5\alpha(k_1^2 + k_2^2 + k_1 k_2 )-3\beta\}}\bigg)\\
	\nonumber
	&-\gamma e^{-\gamma t}\sum_{k_1+k_2=k}\frac{e^{ik_1k_2k\{5\alpha(k_1^2 + k_2^2 + k_1 k_2 )-3\beta\}t}w_{k_1}w_{k_2}}{2k_1k_2\{5\alpha(k_1^2 + k_2^2 + k_1 k_2 )-3\beta\}}\\
	&+e^{-\gamma t}\sum_{k_1+k_2=k}\frac{e^{ik_1k_2k\{5\alpha(k_1^2 + k_2^2 + k_1 k_2 )-3\beta\}t}w_{k_1}\partial_t w_{k_2}}{k_1k_2\{5\alpha(k_1^2 + k_2^2 + k_1 k_2 )-3\beta\}}
	 + e^{\gamma t+i(\alpha k^5-\beta k^3)t} f_k.
\end{align}
Since $w_0=0$, $k_1$ and $k_2$ in the sums above are not zero.
Now, let us define $\widetilde{\mathcal{B}}$ by
$$ \widetilde{\mathcal B}(\phi,\psi)_k=-\frac12\sum_{k_1+k_2=k}\frac{e^{ik_1k_2k\{5\alpha(k_1^2 + k_2^2 + k_1 k_2)-3\beta\}t}\phi_{k_1}\psi_{k_2}}{k_1k_2\{5\alpha(k_1^2 + k_2^2 + k_1 k_2 )-3\beta\}}.$$
Then we rewrite the equation \eqref{morph1} as follows:
	\begin{equation}\label{morphing1}
		\begin{aligned}
			\partial_t(w-e^{-\gamma t}\widetilde{\mathcal{B}}(w,w))_k=&e^{\gamma t+i (\alpha k^5-\beta k^3)t} f_k+ \gamma e^{-\gamma t}\widetilde{\mathcal{B}}(w,w)_k\\
			&\,\,+e^{-\gamma t}\sum_{k_1+k_2=k}\frac{e^{ik_1k_2k\{5\alpha(k_1^2 + k_2^2 + k_1 k_2 )-3\beta\}t}w_{k_1}\partial_t w_{k_2}}{k_1k_2\{5\alpha(k_1^2 + k_2^2 + k_1 k_2 )-3\beta\}}.
		\end{aligned}
	\end{equation}

Applying \eqref{morph} to the sum in \eqref{morphing1}, we see
\begin{align}
\nonumber
&\sum_{k_1+k_2=k}\frac{e^{ik_1k_2k\{5\alpha(k_1^2 + k_2^2 + k_1 k_2 )-3\beta\}t}w_{k_1}\partial_t w_{k_2}}{k_1k_2\{5\alpha(k_1^2 + k_2^2 + k_1 k_2 )-3\beta\}}\\
\nonumber
&\quad\!=-2\widetilde{\mathcal{B}}(w,e^{\gamma t- Lt}f)_k\\
&\quad\!-ie^{-\gamma t}\sum_{k_1+k_2+k_3=k\atop k_2+k_3\ne0}\frac{e^{it\theta(k_1,k_2,k_3)(k_1+k_2)(k_2+k_3)(k_3+k_1)}}{2k_1\{5\alpha(k_1^2 + (k_2+k_3)^2 + k_1 (k_2+k_3) )-3\beta\}}w_{k_1}w_{k_2} w_{k_3},\label{morphing2}
\end{align}
where ${\theta} (k_1,k_2,k_3)=5\alpha(k_1^2 +k_2^2+k_3^2+k_1k_2+k_2k_3+k_3k_1)-3\beta$. 
Combining \eqref{morphing1} and \eqref{morphing2}, we have 
\begin{align}
\nonumber
&\partial_t(w-e^{-\gamma t}\widetilde{\mathcal{B}}(w,w))_k\\
	\nonumber
	&\!=e^{\gamma t+i(\alpha k^5-\beta k^3)t} f_k+ \gamma e^{-\gamma t}\widetilde{\mathcal{B}}(w,w)_k-2e^{-\gamma t}\widetilde{\mathcal{B}}(w,e^{\gamma t- Lt}f)_k\\
	\label{morphing3}
	&\!\quad-ie^{-2\gamma t}\!\sum_{k_1+k_2+k_3=k\atop k_2+k_3\ne0}\!\frac{e^{it \theta(k_1,k_2,k_3)(k_1+k_2)(k_2+k_3)(k_3+k_1)}}{2k_1\{5\alpha(k_1^2 + (k_2+k_3)^2 + k_1 (k_2+k_3) )-3\beta\}}w_{k_1}w_{k_2}w_{k_3}.
\end{align}
Since $\theta (k_1, k_2, k_3 )\ne 0$, the set on which the phase vanishes is the disjoint union of the following sets:
\begin{align*}
	S_1 &= \{k_1+k_2=0, k_3 + k_1=0, k_2 +k_3 \ne 0\}
	=\{k_1=-k, k_2 = k, k_3 = k \}, \\
	S_2&= \{k_1 +k_2 = 0,k_3+k_1 \ne 0,k_2+k_3 \ne 0\}
	= \{k_1=j, k_2 = -j, k_3 = k, |j|\ne|k| \},\\
	S_3 &= \{k_3+k_1=0, k_1 + k_2 \ne 0, k_2 + k_3 \ne 0\}
	= \{k_1=j, k_2 = k, k_3 = -j, |j|\ne|k| \}.
\end{align*}
Let us define $\widetilde {\mathcal R}$ by
$$\widetilde {\mathcal R}(u)_k=-\frac i2\sum_{k_1+k_2+k_3=k\atop(k_1+k_2)(k_2+k_3)(k_3+k_1)\ne0}\frac{e^{it \theta(k_1,k_2,k_3)(k_1+k_2)(k_2+k_3)(k_3+k_1)}u_{k_1}u_{k_2}u_{k_3}} {k_1\{5\alpha(k_1^2 + (k_2+k_3)^2 + k_1 (k_2+k_3) )-3\beta\}}.
$$
Then \eqref{morphing3} becomes
\begin{align}
\nonumber
&\partial_t(w-e^{-\gamma t}\widetilde {\mathcal B}(w,w))_k\\
\nonumber
&=e^{\gamma t+i(\alpha k^5-\beta k^3 )t}f_k+\gamma e^{-\gamma t}\widetilde {\mathcal B}(w,w)_k-2e^{-\gamma t}\widetilde {\mathcal B}(w,e^{\gamma t- Lt}f)_k \\
\nonumber
&\quad+e^{-2\gamma t}\widetilde {\mathcal R}(w)_k-\frac{ie^{-2\gamma t} }{2} \sum_{l=1}^3 \sum_{S_l} \frac{w_{k_1}w_{k_2}w_{k_3} }{k_1\{5\alpha(k_1^2 + (k_2+k_3)^2 + k_1 (k_2+k_3) )-3\beta\}}\\
\nonumber
&=e^{\gamma t+i(\alpha k^5-\beta k^3)t}f_k+\gamma e^{-\gamma t}\widetilde {\mathcal B}(w,w)_k -2e^{-\gamma t}\widetilde {\mathcal B}(w,e^{\gamma t- Lt}f)_k \\
\label{morphgoal}
&\quad+e^{-2\gamma t}\bigg(\widetilde {\mathcal R}(w)_k+\frac{i|w_k|^2w_k}{2(15\alpha k^2-3\beta)k}-iw_k\!\sum_{|j|\ne|k|}\!\frac{|w_j|^2}{j\{5\alpha(k^2-kj+j^2)-3\beta\}}\bigg).
\end{align}
Now, integrating \eqref{morphgoal} from 0 to $t$, we obtain
\begin{align*}
w_k(t)-w_k(0)&=e^{-\gamma t}\widetilde{\mathcal B}(w,w)_k(t)- \widetilde{\mathcal B}(w,w)_k(0)+\int_0^te^{\gamma r+i(\alpha k^5-\beta k^3)r}f_k dr\\
	&\quad +\int_0^t\big(\gamma e^{-\gamma r}\widetilde{\mathcal B}(w,w)_k(r)-2e^{-\gamma r}\widetilde{\mathcal B}(w,e^{\gamma r - Lr}f)_k(r)\big)dr\\
	&\quad +\int_0^t\big(e^{-2\gamma r}\widetilde{\mathcal R}(w)_k(r)+e^{-2\gamma r}\rho(w)_k(r)+e^{-2\gamma r}\sigma(w)_k(r)\big)dr.
\end{align*}
Transforming back to the $u$ variable, we have
\begin{align*}
	&u_k(t)-e^{-\gamma t-i(\alpha k^5-\beta k^3)t}g_k\\
	&\quad=\mathcal B(u,u)_k(t)-e^{-\gamma t-i(\alpha k^5-\beta k^3)t}\mathcal B(u,u)_k(0)\\
	&\qquad+\int_0^te^{-\gamma(t-r)-i(\alpha k^5-\beta k^3)(t-r)}(\gamma\mathcal B(u,u)_k(r)-2\mathcal B(u,e^{\gamma r -Lr}f)_k(r))dr\\
	&\qquad+\int_0^te^{-\gamma(t-r)-i(\alpha k^5-\beta k^3)(t-r)}(f_k+\rho(u)_k(r)+\sigma(u)_k(r)) dr\\
	&\qquad+\int_0^te^{-\gamma(t-r)-i(\alpha k^5-\beta k^3)(t-r)}\mathcal R(u)_k(r)dr
\end{align*}
as desired.
\end{proof}

\subsection{Estimates for $\mathcal{B}, \rho, \sigma$ and $\mathcal{R}$}
Now we obtain the following multilinear estimates, Proposition \ref{goodlem}, which will be used to finish the proof of the smoothing estimates (Proposition \ref{thm1}):
\begin{prop}\label{goodlem}
For $s\le3$, we have
\begin{equation}\label{prop1}
	\|\mathcal B(\phi,\psi)\|_{H^s}\lesssim_{\alpha,\beta}\|\phi\|\|\psi\| \quad {\and} \quad \|\rho(u)\|_{H^s}\lesssim_{\alpha,\beta}\|u\|^3.
\end{equation}
For $s\le2$, we have 
\begin{equation}\label{prop2}
\|\sigma(u)\|_{H^s}\lesssim_{\alpha,\beta}\|u\|^3
\end{equation}
and
\begin{equation}\label{prop3}
	\|\mathcal{R}(u)\|_{X_{\delta}^{s, -1/2+\epsilon}} \lesssim_{\alpha,\beta,\epsilon}\|u\|^3_{X_{\delta}^{0,1/2}}.
\end{equation}
\end{prop}
\subsubsection{Proofs of \eqref{prop1} and \eqref{prop2}}
By symmetry, one can restrict the sum in $\mathcal B$ to $|k_1|\ge|k_2|$, and since $$5\alpha (k_1^2+k_2^2+k_1k_2)-3\beta\sim_{\alpha,\beta}\left(k_2 + \frac{k_1}{2}\right)^2 + \frac{3k_1^2}{4}\ge k_1^2\gtrsim k^2,$$ we have
\begin{align*}
	\|\mathcal B(\phi,\psi)\|_{H^s}&\le\left\||k|^s\sum_{k_1+k_2=k\atop|k_1|\ge|k_2|}
	\frac{|\phi_{k_1}\psi_{k_2}|}{|5\alpha(k_1^2+k_2^2+k_1k_2)-3\beta||k_1k_2|}\right\|_{l_k^2}\\
	&\lesssim_{\alpha,\beta}\left\| \sum_{k_1+k_2=k\atop|k_1|\ge|k_2|}
	\frac{|\phi_{k_1}||\psi_{k_2}||k|^{s-3}}{|k_2|}\right\|_{l_k^2}.
\end{align*}
For $s \le 3$, by Young's inequality and the Cauchy-Schwarz inequality the right side is bounded as
\begin{equation*}
	\left\|\sum_{k_1+k_2=k\atop|k_1|\ge|k_2|}\frac{|\phi_{k_1}||\psi_{k_2}|}{|k_2|}\right\|_{l_k^2}
	\le\|\phi_k\|_{l_k^2}\bigg\|\frac{\psi_k}{|k|}\bigg\|_{l_k^1}\le\|\phi\|\big\||k|^{-1}\big\|_{l_k^2}\|\psi\|\lesssim\|\phi\|\|\psi\|
\end{equation*}
which implies the desired estimate. 
Again for $s\le 3$, using the inclusion property for $l^p$-spaces we obtain the second inequality in \eqref{prop1}:
$$\|\rho(u)\|_{H^s}= \bigg\||k|^s\frac{i|u_k|^2 u_k}{2(15\alpha k^2-3\beta)k}\bigg\|_{l_k^2}\lesssim_{\alpha,\beta}\||u_k|^3\|_{l_k^2}=\||u_k|\|^3_{l_k^6} \le\|u\|^3.$$

To show \eqref{prop2}, we note that the left side of \eqref{prop2} is bounded as
\begin{align*}
\|\sigma(u)\|_{H^s}&=\bigg\||k|^su_k\sum_{j\ne\pm k\atop j\ne0}\frac{u_{j}u_{-j}}{j(5\alpha(k^2-kj+j^2)-3\beta)}\bigg\|_{l_k^2} \\
&\le \sup_{k,j\ne0}\frac{|k|^s}{|j(5\alpha(k^2-kj+j^2)-3\beta)|}\|u\|\sum_{j\ne\pm k\atop j\ne0}|u_j||u_{-j}| \\
&\le \sup_{k,j\ne0}\frac{|k|^s}{|j(5\alpha(k^2-kj+j^2)-3\beta)|}\|u\|^3.
\end{align*}
Here, we used the Cauchy-Schwarz inequality for the last inequality. Note that for $s\le2$, 
$$\sup_{k,j\ne0}\frac{|k|^s}{|j\{5\alpha(k^2-kj+j^2)-3\beta\}|}<\infty$$
since $k^2+j^2-kj=\frac12(k^2+j^2+(k-j)^2)>\frac12k^2$. This completes the proof.

\subsubsection{Proof of \eqref{prop3}}
It suffices to show the statement with a local-in-time function $u$. 
By duality, we may prove 
\begin{align}\label{dualgoal}
\nonumber
\bigg|\int_{\mathbb T^2}\mathcal R(u)(x,t)&h(x,t)dxdt\bigg|\\
&=\bigg|\sum_{k,m}\widehat{\mathcal R}(k,m)\widehat h(-k,-m)\bigg|
\lesssim_{\alpha,\beta,\epsilon}\|u\|_{X^{0,\frac{1}{2}}}^3\|h\|_{X^{-s,\frac12-\epsilon}}.
\end{align}

First, we note that
\begin{align*}&\widehat{\mathcal R}(k,m)\\&=\sum_{k_1+k_2+k_3=k\atop(k_1+k_2)(k_2+k_3)(k_3+k_1)\ne0}\sum_{m_1+m_2+m_3=m}\frac{-i\widehat u(k_1,m_1)\widehat u(k_2,m_2)\widehat u(k_3,m_3)}{2k_1\{5\alpha(k_1^2 +(k_2+k_3)^2 + k_1(k_2+k_3))-3\beta\}}.
\end{align*}
We then let
\begin{align*}
	&\Phi:=\{(k_1,k_2,k_3,k_4)\in \mathbb{Z}^4:k_1+k_2+k_3+k_4=0, (k_1+k_2)(k_2+k_3)(k_3+k_1)\ne0\},\\
	&\Omega:=\{(m_1,m_2,m_3,m_4)\in\mathbb{Z}^4:m_1+m_2+m_3+m_4=0\},
\end{align*}
and set
\begin{align*}
f_i (k_i,m_i) &:=|\widehat u(k_i,m_i)|\langle m_i+\alpha k_i^5-\beta k_i^3\rangle^{\frac{1}{2}},\quad i=1,2,3,\\ f_4(k,m)&:=|\widehat h(k_4,m_4)||k_4|^{-s}\langle m_4+\alpha k_4^5-\beta k_4^3\rangle^{\frac12-\epsilon},\\
	\langle m_n+\alpha k_n^5-\beta k_n^3\rangle&=\max_{i=1,2,3,4}\langle m_i+\alpha k_i^5-\beta k_i^3\rangle,\\ S&:=\{1,2,3,4\}\backslash\{n\}.
	\end{align*}
Then \eqref{dualgoal} follows from
\begin{align}\label{a}
	\nonumber
	\sum_{\Phi,\Omega}&\frac{|k_4|^s\prod_{i=1}^4f_i(k_i,m_i)}{|k_1||5\alpha(k_1^2 +(k_2+k_3)^2 + k_1(k_2+k_3))-3\beta|\prod_{i=1}^4\langle m_i+\alpha k_i^5-\beta k_i^3\rangle^{\frac12-\epsilon}}\\
	&\qquad \qquad \qquad \qquad \qquad \qquad \qquad \qquad \qquad\qquad\qquad\quad \qquad\lesssim_{\alpha,\beta,\epsilon}\prod_{i=1}^4\|f_i\|.
	\end{align}

Since
\begin{align*}
	|5\alpha(k_1^2 +(k_2+k_3)^2 + k_1(k_2+k_3))-3\beta|&\sim_{\alpha,\beta}\left(k_1 + \frac{k_2+k_3}2\right)^2 + \frac{3(k_2+k_3)^2}4\\
	&=\frac{|k_1-k_4|^2}{4} + \frac{3|k_1+k_4|^2}{4}\\
	&\gtrsim |k_1-k_4||k_1+k_4| \gtrsim |k_4|,
\end{align*}
the left side of \eqref{a} is bounded by 
\begin{equation}\label{dualgoal2}
 \sum_{\Phi,\Omega}\frac{|k_4|^{s-1}\prod_{i=1}^4f_i(k_i,m_i)}{|k_1|\prod_{i=1}^4\langle m_i+\alpha k_i^5-\beta k_i^3\rangle^{\frac12-\epsilon}}\lesssim_{\alpha,\beta}\sum_{\Phi,\Omega}\frac{|k_1k_2k_3k_4|^{-\epsilon}\prod_{i=1}^4f_i(k_i,m_i)}{\prod_{i\in S}\langle m_i+\alpha k_i^5-\beta k_i^3\rangle^{\frac12+\epsilon}};
\end{equation}
here, we used the inequality
\begin{equation}\label{c}
	\frac{|k_4|^{s-1}}{|k_1|\prod_{i=1}^4\langle m_i+\alpha k_i^5-\beta k_i^3\rangle^{\frac12-\epsilon}} \lesssim_{\alpha,\beta}\frac{|k_1 k_2 k_3 k_4|^{-\epsilon}}{\prod_{i\in S}\langle m_i+\alpha k_i^5-\beta k_i^3\rangle^{\frac12+\epsilon}}
\end{equation}
for $s\le2$ and $\epsilon>0$ sufficiently small,  which will be shown later.
By symmetry let us now assume $m=1$. 
We may eliminate $|k_1|^{-\epsilon}$ from the right side of \eqref{dualgoal2}.
We then use Lemma \ref{l6} below to confirm \eqref{a}. 
Indeed, by the convolution structure, using Plancherel's theorem and H\"older's inequality as well as Lemma \ref{l6} 
we bound the right side of  \eqref{dualgoal2} as	
\begin{align*}
	\sum_{\Phi,\Omega}&\frac{|k_2k_3k_4|^{-\epsilon}\prod_{i=1}^4f_i(m_i,k_i)}{\prod_{i\in S}\langle m_i+\alpha k_i^5-\beta k_i^3\rangle^{\frac12+\epsilon}}\\
	&\qquad\qquad \qquad \le \int_{\mathbb T^2}\left|\check{f_1}(x,t)\prod_{i=2}^4\bigg(\frac{f_i|k_i|^{-\epsilon}}{\langle m_i+\alpha k_i^5-\beta k_i^3\rangle^{\frac12+\epsilon}}\bigg)^\vee(x,t)\right|dxdt\\ &\qquad\qquad\qquad\le\|f_1\|_{L^2}\prod_{i=2}^4\bigg\|\bigg(\frac{f_i|k|^{-\epsilon}}{\langle m+\alpha k^5-\beta k^3\rangle^{\frac12+\epsilon}}\bigg)^\vee\bigg\|_{L^6}\lesssim_{\alpha,\beta,\epsilon}\prod_{i=1}^4\|f_i\|_{L^2},
\end{align*}
as desired.

\begin{lem}\label{l6}
	For any $\epsilon>0$ and $b>1/2$, we have
	\begin{equation*}
		\|\chi_{[0,\delta]}(t)u\|_{L_{t,x}^6(\mathbb R\times\mathbb T)}\lesssim_{\alpha,\beta,\epsilon,b}\|u\|_{X_\delta^{\epsilon,b}}.
\end{equation*}
\end{lem}

We postpone the proof of this lemma to the next subsection.
Instead we first show \eqref{c}.
Using $m_1+m_2+m_3+m_4=0$ and $k_1+k_2+k_3+k_4=0$, we have
\begin{align*}
\sum_{i=1}^4&m_i+\alpha k_i^5-\beta k_i^3\\
&=-(k_1+k_2)(k_2+k_3)(k_3+k_1)\{5\alpha(k_1^2+k_2^2+k_3^2+k_1k_2+k_2k_3+k_3k_1)-3\beta\}.
\end{align*}
Thus,
\begin{align*}
\langle m_n&+\alpha k_n^5-\beta k_n^3\rangle \\ &\gtrsim|k_1+k_2||k_2+k_3||k_3+k_1||5\alpha(k_1^2+k_2^2+k_3^2+k_1k_2+k_2k_3+k_3k_1)-3\beta|.
	\end{align*}
By using this inequality, we then have
\begin{align*}
&\prod_{i=1}^4\langle m_i+\alpha k_i^5-\beta k_i^3\rangle^{\frac12-\epsilon}\\
&=\langle m_n+\alpha k_n^5-\beta k_n^3\rangle^{\frac12-\epsilon}\prod_{i\in S}\langle m_i+\alpha k_i^5-\beta k_i^3\rangle^{-2\epsilon}\langle m_i\alpha k_i^5-\beta k_i^3\rangle^{\frac12+\epsilon}\\
&\ge\langle m_n+\alpha k_n^5-\beta k_n^3\rangle^{\frac12-7\epsilon}\prod_{i\in S}\langle m_i+\alpha k_i^5-\beta k_i^3\rangle^{\frac12+\epsilon}\\
&\gtrsim\big(|k_1+k_2||k_2+k_3||k_3+k_1||5\alpha(k_1^2+k_2^2+k_3^2+k_1k_2+k_2k_3+k_3k_1)-3\beta|\big)^{\frac12-7\epsilon}\\&\qquad\qquad\qquad\qquad\qquad\qquad\qquad\qquad\qquad\qquad\qquad\qquad\qquad\prod_{i\in S}\langle m_i+\alpha k_i^5-\beta k_i^3\rangle^{\frac12+\epsilon}.
\end{align*}
Here, since
	\begin{align*}
	|5\alpha (k_1^2+k_2^2+k_3^2&+k_1k_2+k_2k_3+k_3k_1)-3\beta|\\
	&\qquad \qquad\qquad\gtrsim_{\alpha,\beta}|k_1+k_2|^2 + |k_2+k_3|^2 + |k_3+k_1|^2 \\
	&\qquad\qquad\qquad\gtrsim (|k_1+k_2||k_2+k_3||k_3+k_1|)^{2/3},
\end{align*}
we conclude
	\begin{align}
\nonumber\prod_{i=1}^4&\langle m_i+\alpha k_i^5-\beta k_i^3\rangle^{\frac12-\epsilon}\\
&\gtrsim_{\alpha,\beta}(|k_1+k_2||k_2+k_3||k_3+k_1|)^{\frac{5}{6}-\frac{35}{3}\epsilon}\prod_{i\in S}\langle m_i+\alpha k_i^5-\beta k_i^3\rangle^{\frac12+\epsilon}.\label{dualwork1}	
	\end{align}
The desired estimate \eqref{c} follows now by combining \eqref{dualwork1} with 
\begin{equation}\label{dualwork3}
	\frac{|k_4|^{s-1}}{|k_1|(|k_1+k_2||k_2+k_3||k_3+k_1|)^{\frac{5}{6}-\frac{35}{3}\epsilon}}\lesssim|k_1k_2k_3k_4|^{-\epsilon}.
\end{equation}

To show \eqref{dualwork3},
using $|k_1||k_1+k_2|\gtrsim|k_2|$ and $|k_1||k_1+k_3|||k_2+k_3|\gtrsim|k_2|$, and by symmetry of $k_2$, $k_3$, we first see
\begin{align}
	\nonumber
|k_1|(|k_1+k_2||k_2&+k_3||k_3+k_1|)^{\frac{5}{6}-\frac{47}{3} \epsilon}\\ 
	\nonumber
&=|k_1|^{\frac{1}{2}}|k_1+k_2|^{\frac{5}{6}-\frac{47}{3} \epsilon}|k_1|^{\frac12}|k_2+k_3|^{\frac{5}{6}-\frac{47}{3} \epsilon}|k_3+k_1|^{\frac{5}{6}-\frac{47}{3} \epsilon}\\
\label{d}
&\ge \max\{|k_1|,|k_2|,|k_3|\} \gtrsim|k_4|.
\end{align}
Since all factors in the left side of \eqref{d} are nonzero and $k_1+k_2+k_3+k_4=0$, we also see 
\begin{equation}\label{dualwork2}
|k_1+k_2||k_2+k_3||k_3+k_1|\gtrsim|k_i|,\quad i=1,2,3,4.
\end{equation}
By \eqref{d} and \eqref{dualwork2}, for $s\le2$ and sufficiently small $\epsilon$, we get
$$\frac{|k_4|^{s-1}}{|k_1|(|k_1+k_2||k_2+k_3||k_3+k_1|)^{\frac{5}{6}-\frac{47}{3}\epsilon}}\lesssim1,$$
which implies \eqref{dualwork3}.

\subsubsection{Proof of Lemma \ref{l6}}
The argument is similar to that for KdV in \cite B. Regard $\|\chi_{[0,\delta]}(t)u\|_{L_{t,x}^6(\mathbb R\times\mathbb T)}$ as $\|u\|_{L_{t,x}^6(\mathbb T^2)}$.
By the Fourier inversion formula and a change of variables, we have
\begin{align*}
	u(x,t)&=\int_\mathbb R\sum_k\widehat u(k,\tau)e^{ikx+i\tau t}d\tau\\
	&=\int_\mathbb R\sum_k\widehat u(k,\tau-\alpha k^5+\beta k^3)e^{ikx+i\tau t}e^{-i(\alpha k^5-\beta k^3)t}d\tau=\int_\mathbb Re^{i\tau t}e^{Lt}h_\tau d\tau,
\end{align*}
where $h_\tau=\sum_k\widehat u(k,\tau-\alpha k^5+\beta k^3)e^{ikx}.$
If we have
\begin{equation} \label{h}
	\|e^{Lt}h\|_{L_{t,x}^6(\mathbb T^2)}\lesssim_{\alpha,\beta,\epsilon}\|h\|_{H^\epsilon},
\end{equation}
for any $h\in H^\epsilon$, $\epsilon>0$, by the Minkowski inequality, \eqref{h} and the Cauchy-Schwarz inequality we obtain
\begin{align*}
	\|u\|_{L^6_{x,t}(\mathbb{T}^2)}&\le\int_\mathbb R\|e^{Lt}h_\tau\|_{L^6_{x,t}(\mathbb{T}^2)}d\tau\\
	&\lesssim_{\alpha,\beta,\epsilon}\int_\mathbb R\|h_\tau\|_{H^\epsilon}d\tau\\
	&\lesssim\left(\int_\mathbb R\langle\tau\rangle^{-2b}d\tau\right)^\frac12\left(\int_\mathbb R\langle\tau\rangle^{2b}\|h_\tau\|_{H^\epsilon}^2d\tau\right)^\frac12\lesssim_b\|u\|_{X_{\delta}^{\epsilon,b}}
\end{align*}
for $b>1/2$. 

It remains to prove \eqref{h}.
We first consider ${P_K}h:=\sum_{|k|\sim K}e^{ikx}h_k$ with $K$ dyadic. 
Then, 
\begin{align*}
	(e^{Lt}P_Kh)^3&=\sum_{|k_1|, |k_2|, |k_3|\sim K}e^{i\{(-\alpha(k_1^5+k_2^5+k_3^5)+\beta(k_1^3+k_2^3+k_3^3))t+(k_1+k_2+k_3)x\} } h_{k_1}h_{k_2}h_{k_3}\\
	&=\sum_{|k_1|, |k_2|, |k_3|\sim K}e^{i\kappa t}e^{i(k_1+k_2+k_3)x}h_{k_1}h_{k_2}h_{k_3},
\end{align*}
where $\kappa=\kappa(k_1,k_2,k_3):=-\alpha(k_1^5+k_2^5+k_3^5)+\beta(k_1^3+k_2^3+k_3^3)$. 
Hence,
\begin{align}\label{i}
	\nonumber
	&\|e^{Lt}P_Kh\|_{L_{t,x}^6(\mathbb T^2)}^6\\
	\nonumber
	&=\iint_{\mathbb T^2}\sum_{k_1,k_2,k_3}e^{i((k_1+k_2+k_3)x+\kappa t)}h_{k_1}h_{k_2}h_{k_3} \overline{\sum_{\tilde k_1,\tilde k_2, \tilde k_3}e^{i((\tilde k_1+ \tilde k_2+\tilde k_3)x+\tilde \kappa t)}h_{\tilde k_1}h_{\tilde k_2}h_{\tilde k_3}}dxdt\\
	\nonumber
	&=\sum_{k_1+k_2+k_3=\tilde k_1+\tilde k_2+\tilde k_3\atop\kappa=\tilde\kappa}
	h_{k_1}h_{k_2}h_{k_3} \overline{h_{\tilde k_1}h_{\tilde k_2}h_{\tilde k_3}}\\
	&=\sum_{p,q}\bigg|\sum_{(k_1,k_2,k_3)\in A_{p,q}}h_{k_1}h_{k_2}h_{k_3}\bigg|^2, 	
\end{align}
where $A_{p,q}= \{(k_1,k_2,k_3):|k_i|\sim K, k_1+k_2+k_3=p, \kappa(k_1,k_2,k_3)=q\}$.
We now claim that $|A_{p,q}| \lesssim K^\epsilon$ for any $\epsilon>0$.
From the equality
$$\frac{q+\alpha p^5-\beta p^3}{5\alpha(k_1^2 + k_2^2+  k_3^2+ k_1k_2 + k_2k_3 + k_3k_1)-3\beta}=(k_1 +k_2)(k_2+k_3)(k_3+k_1)$$
the left-hand side must be an integer. Thus if $\alpha$ is not a rational multiple of $\beta$, there is at most one possible value for the denominator; we can choose a value that solely depends on $\alpha$ and $\beta$ than which such cannot be smaller. Set $C_{\alpha,\beta}$ as the reciprocal of the denominator. If $\alpha$ is a rational multiple of $\beta$, choose $C_{\alpha,\beta}\ne0$ such that both $C_{\alpha,\beta}\alpha$ and $C_{\alpha,\beta}\beta$ are integers. Either way $C_{\alpha,\beta}(q+\alpha p^5-\beta p^3)$ also has to be an integer. Then we have
\begin{align*}
&C_{\alpha,\beta}(q+\alpha p^5-\beta p^3)\\
&=(k_1 +k_2)(k_2+k_3)(k_3+k_1)C_{\alpha,\beta}(5\alpha(k_1^2 + k_2^2+  k_3^2+ k_1k_2 + k_2k_3 + k_3k_1)-3\beta).
\end{align*}
Since the number of divisors of integer $n$ is $o(n^\epsilon)$ for any $\epsilon>0$, we see that the triple $(k_1+k_2, k_2+k_3, k_3+k_1)$ or equivalently $(k_1,k_2,k_3)$ has no more than $(C_{\alpha,\beta}K)^{\epsilon}$ choices.
Now by \eqref{i} and the Cauchy-Schwarz inequality, we conclude that
\begin{align*}
	\|e^{Lt}h\|_{L_{t,x}^6(\mathbb T^2)}&
	\le\sum_K\|e^{Lt}P_Kh\|_{L^6_{t,x}} \\
	&\le\sum_K\bigg(\sum_{p,q}\Big(\sum_{(k_1,k_2,k_3)\in A_{p,q}}|h_{k_1}h_{k_2}h_{k_3}|^2\Big)\Big(\sum_{(k_1,k_2,k_3)\in A_{p,q}}1\Big)\bigg)^{\frac16}\\
	&\lesssim_{\alpha,\beta}\sum_K K^{\epsilon/6} \bigg(\sum_{|k_1|,|k_2|,|k_3|\sim K}|h_{k_1}h_{k_2}h_{k_3}|^2\bigg)^{\frac16}\\
	&\le\bigg(\sum_KK^{-\epsilon/3}\bigg)^{\frac12}\bigg(\sum_KK^{\epsilon}\|P_Kh\|_{L_x^2}^2\bigg)^{\frac12}\lesssim_\epsilon\|h\|_{H^{\epsilon}}.
\end{align*}

\subsection{The last step of the proof}
Now we finish the proof of Proposition \ref{thm1}.
Applying the estimates in Proposition \ref{goodlem} to the equation \eqref{morphback}, we obtain for $s\le2$
\begin{align}
\nonumber
\|u(t)&-e^{-\gamma t}e^{Lt}g\|_{H^s}\\
\nonumber
&\qquad\lesssim_{\alpha,\beta,\gamma}\|u(t)\|^2+\|g\|^2+\int_0^t\big(\|u(r)\|^2+\|u(r)\|\|e^{\gamma r-Lr}f\|+\|u(r)\|^3\big)dr\\
\nonumber
&\qquad\quad+\bigg\|\int_0^t e^{(-\gamma + L)(t-r)}f dr \bigg\|_{H^s} +\bigg\|\int_0^te^{(-\gamma+L)(t-r)}\mathcal R(u)(r)dr\bigg\|_{H^s}\\
\nonumber
&\qquad\lesssim_{\alpha,\beta}\|u(t)\|^2+\|g\|^2+\int_0^t\big(\|u(r)\|^2+e^{\gamma r}\|u(r)\|\|f\|+\|u(r)\|^3\big)dr\\
\label{int}
&\qquad\quad+\|f\|+\bigg\|\int_0^te^{-\gamma(t-r)}e^{L(t-r)}\mathcal R(u)(r)dr\bigg\|_{H^s}.
\end{align}
Here we used the fact that for $s\le 2$,
\begin{align*}
\bigg\|\int_0^t e^{(-\gamma +L)(t-r)}f dr \bigg\|_{H^s}&\sim
\bigg\||k|^s\frac{f_k}{\gamma+i\alpha k^5-i\beta k^3}\left(1-e^{(-\gamma-i\alpha k^5+i\beta k^3)t}\right)\bigg\|_{l_k^2} \\
&\lesssim\bigg\||k|^s\frac{|f_k|}{|\alpha k^5-\beta k^3|}\bigg\|_{l_k^2}\lesssim_{\alpha,\beta}\|f\|.
\end{align*}

Now we handle the last term in \eqref{int} using the embedding 
$$X_{\delta}^{s,b}\subset L_t^\infty H_x^s([0,\delta]\times \mathbb{T}), \quad b>1/2,$$
and the following lemma (See Lemma 3.3 and its proof in \cite{ET2}).
\begin{lem}\label{ri}
	Let $b>1/2$. 
	Then for any $\delta<1$,
$$\bigg\|\int_0^te^{-\gamma(t-r)}e^{L(t-r)}F(r)dr\bigg\|_{X_\delta^{s,b}}\lesssim_{\gamma,b}\|F\|_{X_\delta^{s,b-1}}.$$
\end{lem}

For $b>1/2$ and $t<\delta$ with $\delta$ given in Theorem \ref{thm2}, we have 
\begin{align}\label{rinitiate3}
	\nonumber
	\bigg\|\int_0^te^{-\gamma(t-r)}e^{L(t-r)}\mathcal R(u)(r)dr\bigg\|_{H^s}&\le\bigg\|\int_0^te^{-\gamma(t-r)}e^{L(t-r)}\mathcal R(u)(r)dr\bigg\|_{L_{t\in[0,\delta]}^\infty H_x^s}\\
	\nonumber
	&\lesssim\bigg\|\int_0^te^{-\gamma(t-r)}e^{L(t-r)}\mathcal R(u)(r)dr\bigg\|_{X_\delta^{s,b}}\\
	&\lesssim_{\gamma,b}\|\mathcal R(u)\|_{X_\delta^{s,b-1}}\lesssim_{\alpha,\beta,b}\|u\|_{X_\delta^{0,1/2}}^3.
\end{align}
Here we used \eqref{prop3} for the last inequality.

Combining \eqref{int} and \eqref{rinitiate3}, we see that for $t<\delta$,
\begin{align*}
\|&u(t)-e^{-\gamma t}e^{Lt}g\|_{H^s}\\
&\lesssim_{\alpha,\beta,\gamma}\|u(t)\|^2+\|g\|^2+\int_0^t\big(\|u(r)\|^2+\|u(r)\|\|f\|+\|u(r)\|^3\big)dr+\|f\|+\|u\|_{X_\delta^{0,1/2}}^3.
\end{align*}
Now, fix $t$ large. 
For $r\le t$, we have $\|u(r)\|\le\|g\|+\|f\|/\gamma$ from \eqref{absorb}.
Hence, from \eqref{bd} we see for any $j\le t/\delta$
$$\|u\|_{X_{[(j-1)\delta,j\delta]}^{0,1/2}}\lesssim\|u((j-1)\delta)\|\le\|g\|+\|f\|/\gamma.$$
Consequently, we have
$$\|u(j\delta)-e^{j\delta(L-\gamma)}u((j-1)\delta)\|_{H^s}\lesssim1$$
if we choose some $\alpha>0$ so that $\delta\sim\langle\|g\|+\|f\|/\gamma\rangle^{-\alpha}$ with the implicit constant depending on $\|f\|$, $\|g\|$ and $\gamma$.
Using this we obtain (with $J=t/\delta$)
\begin{align*}
	\|u(J\delta)-e^{J\delta(L-\gamma)}g\|_{H^s}&\le\sum_{j=1}^J\|e^{(J-j)\delta(L-\gamma)}u(j\delta)-e^{(J-j+1)\delta(L-\gamma)}u((j-1)\delta)\|_{H^s}\\
	&=\sum_{j=1}^Je^{-(J-j)\delta\gamma}\|u(j\delta)-e^{\delta(L-\gamma)}u((j-1)\delta)\|_{H^s}\\
	&\lesssim\sum_{j=1}^Je^{-(J-j)\delta\gamma}\lesssim\frac1{1-e^{-\delta\gamma}}.
\end{align*}
This completes the proof of \eqref{j}. Taking similar steps to above one can easily show the continuity of $u(t)-e^{-\gamma t}e^{Lt}g$ in $H^s$.

\section{Proofs of Lemmas \ref{lem2} and \ref{lem}}\label{sec6}
In this section, we provide the proofs of Lemmas \ref{lem2} and \ref{lem} which follow the arguments for the KdV equation in \cite B.

\subsection{Proof of Lemma \ref{lem2}}
Estimates \eqref{w1} and \eqref{w2} are standard but here we provide a detailed proof for the convenience of the reader.
Without loss of generality we set $s=0$. 

Let $\eta\in C_0^{\infty}(\mathbb{R})$ be a smooth function supported on $[-2,2]$ with $\eta=1$ on $[-1,1]$.
We first prove \eqref{w1}. 
By the definition of $X^{s,b}$-norm, we see 
\begin{align*}
\big\|\eta(t)e^{Lt}g\big\|_{X^{0,\frac12}}
=\bigg(\sum_{k\in\mathbb Z}\int_\mathbb R\langle\tau+\alpha k^5-\beta k^3\rangle|g_k\widehat\eta(\tau+\alpha k^5-\beta k^3)|^2d\tau\bigg)^\frac12.
\end{align*}
Since $\widehat\eta$ decays faster than any finite power of $\langle\tau+\alpha k^5-\beta k^3\rangle^{-1}$, we have
\begin{equation}\label{lem213}
	\|\eta(t)e^{Lt}g\|_{X^{0,\frac12}}\lesssim\bigg(\sum_{k\in\mathbb Z}|g_k|^2\bigg)^\frac12=\|g\|
\end{equation}
with the implicit constant depending only on the choice of $\eta$.
Similarly the $l_k^2L_\tau^1$-norm of the  $\mathcal F\big(\eta(t)e^{Lt}g(x)\big)$ is
$$\bigg(\sum_{k\in\mathbb Z}\bigg|g_k\int_\mathbb R|\widehat\eta(\tau+\alpha k^5-\beta k^3)|d\tau\bigg|^2\bigg)^\frac12\lesssim\|g\|,$$
which, along with \eqref{lem213}, proves \eqref{w1}.

Now we prove \eqref{w2}. We present a neater proof following \cite{CKSTT1}. By applying a smooth cutoff one may assume that $F$ is supported on $\mathbb T\times[-3,3]$. 
Let us set $a(t):=\text{sgn}(t)\tilde\eta(t)$, where $\tilde\eta\in C_0^{\infty}$ is supported on $[-10,10]$ with $\tilde\eta(t)=1$ on $[-5,5]$. From the identity
$$\chi_{[0,t]}(t')=\frac{\text{sgn}(t)}{2}(a(t')+a(t-t'))$$
for all $t\in[-2,2]$ and $t'\in[-3,3]$, we may write $\eta(t)\int_0^te^{L(t-t')}F(t')dt'$ as a linear combination of
\begin{equation}\label{finfine1}\eta(t)e^{Lt}\int_\mathbb Ra(t')e^{-Lt'}F(t')dt'
\end{equation}
and
\begin{equation}\label{finfine2}\eta(t)\int_\mathbb Ra(t-t')e^{L(t-t')}F(t')dt'.
\end{equation}
For the contribution of \eqref{finfine1}, by \eqref{w1}, it suffices to show that
$$\left\|\int_\mathbb Ra(t')e^{-Lt'}F(t')dt'\right\|_{H^s}\lesssim\|F\|_{Z^s}.$$
We note that the space Fourier transform of $\int_\mathbb Ra(t')e^{-Lt'}F(t')dt'$ is written as
\begin{align*}
	\mathcal F_x\left(\int_\mathbb Ra(t')e^{-Lt'}F(x,t')dt'\right)(k)&=\int_\mathbb Ra(t')e^{i(\alpha k^5-\beta k^3)t'}\int_\mathbb Re^{i\tau t'}\widehat F(k,\tau)d\tau dt'\\
	&=-2\pi\int_\mathbb R\widehat a(\tau+\alpha k^5-\beta k^3)\widehat F(k,\tau)d\tau.
\end{align*}
Since one has the easily verified bound
\begin{equation}\label{finfine3}
	|\widehat a|\lesssim\langle\cdot\rangle^{-1},
\end{equation}
we obtain 
\begin{align*}
	\left\|\int_\mathbb Ra(t')e^{-Lt'}F(t')dt'\right\|_{H^s}&=\left\|\langle k\rangle^s\int_\mathbb R\widehat a(\tau+\alpha k^5-\beta k^3)\widehat F(k,\tau)d\tau\right\|_{l_k^2}\\
	&\lesssim \left\|\frac{\langle k\rangle^s\widehat F(k,\tau)}{\langle\tau+\alpha k^5-\beta k^3\rangle}\right\|_{l_k^2L_\tau^1}\le\|F\|_{Z^s}.
\end{align*}
Now consider the contribution of \eqref{finfine2}. We may ignore $\eta(t)$. Since the space-time Fourier transform of $\int_\mathbb Ra(t-t')e^{L(t-t')}F(t')dt'$ is $2\pi\widehat a(\tau+\alpha k^5-\beta k^3)\widehat F(k,\tau)$, by the definitions of $Y^s$- and $Z^s$-norms as well as \eqref{finfine3}, the claim follows as
$$\left\|\mathcal F^{-1}\left(\widehat a(\tau+\alpha k^5-\beta k^3)\widehat F(k,\tau)\right)\right\|_{Y^s}\lesssim\left\|\mathcal F^{-1}\left(\frac{\widehat F(k,\tau)}{\langle\tau+\alpha k^5-\beta k^3\rangle}\right)\right\|_{Y^s}=\|F\|_{Z^s}.$$

\subsection{Proof of Lemma \ref{lem}}\label{subsec5.2}
First, \eqref{w4} is a standard estimate for dispersive equations; see, for example, \cite{Ta,AKS}. To show \eqref{w5}, by \eqref{w4} it suffices to consider the second part of the $Z_\delta^s$-norm. This follows from
\begin{align*}
	\bigg\|\frac{\langle k\rangle^s\widehat{u\eta(\cdot/\delta)}(k,\tau)}{\langle\tau+\alpha k^5-\beta k^3\rangle}\bigg\|_{l_k^2L_\tau^1}
	&\lesssim\big\|\langle k\rangle^s\widehat{u\eta(\cdot/\delta)}(k,\tau)\big\|_{l_k^2L_\tau^{\infty-}}\bigg\|\frac1{\langle\tau+\alpha k^5-\beta k^3\rangle}\bigg\|_{L_\tau^{1+\epsilon}}\\
	&\lesssim_\epsilon\|\delta\widehat\eta(\delta\tau)\|_{L_\tau^{\infty-}}\big\|\langle k\rangle^s\widehat u(k,\tau)\big\|_{l_k^2L_\tau^1} \\
	&\lesssim\delta^{1-}\big\|\langle k\rangle^s\widehat u(k,\tau)\big\|_{l_k^2L_\tau^1}.
\end{align*}
In the first inequality we used H\"older's inequality, and in the second, Young's inequality in the $\tau$ variable. 
The estimate \eqref{157} follows in a similar manner with the only difference that one does not need Young's inequality for $f$ is independent of time.

Now it remains to prove \eqref{w3}. 
To do so, we need to show the following two inequalities:
\begin{align}
\label{fflem1}
\|\partial_x(u_1u_2)\|_{X_\delta^{s,-\frac12}}&\lesssim\|u_1\|_{X_\delta^{s,\frac3{10}}}\|u_2\|_{X_\delta^{s,\frac3{10}}},\\
\label{fflem2}
\bigg\|\frac{\langle k\rangle^s\mathcal F(\partial_x(\eta u_1u_2))}{\langle\tau+\alpha k^5-\beta k^3\rangle}\bigg\|_{l_k^2L_\tau^1}&\lesssim\|u_1\|_{X_\delta^{s,\frac3{10}}}\|u_2\|_{X_\delta^{s,\frac3{10}}}.
\end{align}
Here we assume $\int_\mathbb Tu_1(x,t)dx=\int_\mathbb Tu_2(x,t)dx=0.$
\begin{rem}
One can also bound up to $X_\delta^{s,-\frac3{10}}$-norm in \eqref{fflem1} and $\langle\tau+\alpha k^5-\beta k^3\rangle^{\frac45+}$ in the denominator in the $l_k^2L_\tau^1$-norm in \eqref{fflem2}. However, for simplicity we retain the definition form of the $Y^s$-norm. See Remark \ref{remref} for details.
\end{rem} 

Note that
\begin{equation}\label{fflemnote1}
	|\mathcal F(\partial_x(u_1u_2))(k,\tau)|\le|k|\sum_{k_1\ne0}\int_\mathbb R|\widehat{u_1}(k_1,\tau_1)||\widehat{u_2}(k-k_1,\tau-\tau_1)|d\tau_1.
\end{equation}
We shall define
\begin{equation}\label{fflemnote2}
	c_j(k,\tau):=|k|^s\langle\tau+\alpha k^5-\beta k^3\rangle^\frac14|\widehat{u_j}(k,\tau)|
\end{equation}
for $j=1,2$.

\subsection*{Proof of \eqref{fflem1}}
From \eqref{fflemnote1} and \eqref{fflemnote2},
\begin{align}\label{fflem1l1}
&\frac{|k|^s|\mathcal F(\partial_x(u_1u_2))(k,\tau)|}{\langle\tau+\alpha k^5-\beta k^3\rangle^\frac12}\\
\nonumber&\le\sum_{k_1}\int_\mathbb R\frac{|k|^{s+1}|k_1|^{-s}|k-k_1|^{-s}c_1(k_1,\tau_1)c_2(k-k_1,\tau-\tau_1)}{\langle\tau\!\!\;+\!\!\;\alpha k^5\!\!\;-\!\!\;\beta k^3\rangle^\frac12\langle\tau_1\!\!\;+\!\!\;\alpha k_1^5\!\!\;-\!\!\;\beta k_1^3\rangle^\frac14\langle\tau\!\!\;-\!\!\;\tau_1\!\!\;+\!\!\;\alpha(k\!\!\;-\!\!\;k_1)^5\!\!\;-\!\!\;\beta(k\!\!\;-\!\!\;k_1)^3\rangle^\frac14}d\tau_1.
\end{align}
By assumptions on $u_1$ and $u_2$, we have $c_i(0,\tau)=0$, $i=1,2$, so in \eqref{fflem1l1} we may assume $k_1\ne0$ and $k-k_1\ne0$ as well as $k\ne0$. 
Since
\begin{align}
\nonumber&|(\tau+\alpha k^5-\beta k^3)-(\tau_1+\alpha k_1^5-\beta k_1^3)-(\tau-\tau_1+\alpha(k-k_1)^5-\beta(k-k_1)^3)|\\
\label{powerbound}&=|k_1(k-k_1)k(5\alpha(k^2+k_1^2-kk_1)-3\beta)|\gtrsim_{\alpha,\beta}k^4,
\end{align}
it suffices to sum up the contributions of the following cases:
	\begin{align}
		\label{fflem1c1}&|\tau+\alpha k^5-\beta k^3|\gtrsim k^4,\\
		\label{fflem1c2}&|\tau_1+\alpha k_1^5-\beta k_1^3|\gtrsim k^4,\\
		\label{fflem1c3}&|\tau-\tau_1+\alpha(k-k_1)^5-\beta(k-k_1)^3|\gtrsim k^4.
	\end{align}
	
We begin with the first case \eqref{fflem1c1}. 
From \eqref{fflem1c1} and
\begin{equation}\label{fflem1l2}
	|k|^s\le2^s|k_1|^s|k-k_1|^s,
\end{equation}
the right side of \eqref{fflem1l1} is bounded by
\begin{equation}\label{fflem1l3}
C_{\alpha,\beta,s}\sum_{k_1}\int_\mathbb R\frac{c_1(k_1,\tau_1)c_2(k-k_1,\tau-\tau_1)}{|k|\langle\tau_1+\alpha k_1^5-\beta k_1^3\rangle^\frac14\langle\tau-\tau_1+\alpha(k-k_1)^5-\beta(k-k_1)^3\rangle^\frac14}d\tau_1.
\end{equation}
Define for $j=1,2$
\begin{equation}\label{tj}
T_j(x,t):=\sum_k\int_\mathbb Re^{i(kx+\tau t)}\frac{c_j(k,\tau)}{\langle\tau+\alpha k^5-\beta k^3\rangle^\frac14}d\tau
\end{equation}
and $\eqref{fflem1l3}\le\widehat{T_1T_2}(k,\tau).$
Hence the contribution to the left side of \eqref{fflem1} is bounded as
$$\bigg(\sum_{k\ne0}\int_\mathbb R\big|\widehat{T_1T_2}(k,\tau)\big|^2d\tau\bigg)^\frac12\sim\|T_1T_2\|_{L_x^2L_{t,\text{loc}}^2}\lesssim\|T_1\|_{L_x^4L_{t,\text{loc}}^4}\|T_2\|_{L_x^4L_{t,\text{loc}}^4}.$$
Here we have a local estimate for we assume $u_1$ and $u_2$ to be local-in-time functions; to be specific one may write $u_j'(x,t)=u_j(x,t)\eta(t)$ for a suitable localisation function $\eta$ and we accordingly get $\widehat{u_j'}=\widehat{u_j}*\widehat\eta$ as well as $\eqref{fflem1l3}=\widehat{T_1T_2}*\widehat\eta=\widehat{T_1T_2\eta}$. With this remark, we shall use the estimate
	\begin{equation}\label{t2propeq}
		\|\chi_{[0,\delta]}(t)f\|_{L^4(\mathbb T\times\mathbb R)}\lesssim_{\alpha,\beta}\|f\|_{X_{\delta}^{0,\frac3{10}}}
	\end{equation}
whose proof is given at the end of this section. Using \eqref{t2propeq}, we have
\begin{equation}\label{fflem1l4}
	\|T_j\|_4\lesssim_{\alpha,\beta}\|T_j\|_{X^{0,\frac3{10}}}=\bigg\|\frac{c_j}{\langle\tau+\alpha k^5-\beta k^3\rangle^{\frac14-\frac3{10}}}\bigg\|_{L^2}=\|u_j\|_{X^{s,\frac3{10}}}
\end{equation}
for $j=1,2$. 

For the second case \eqref{fflem1c2}, again by \eqref{fflem1l2}, the right side of \eqref{fflem1l1} is bounded by
\begin{equation*}
\sum_{k_1}\int_\mathbb R\frac{c_1(k_1,\tau_1)c_2(k-k_1,\tau-\tau_1)}{\langle\tau+\alpha k^5-\beta k^3\rangle^\frac12\langle\tau-\tau_1+\alpha(k-k_1)^5-\beta(k-k_1)^3\rangle^\frac14}d\tau_1.
\end{equation*}
	Defining
	$$T_3(x,t):=\sum_k\int_\mathbb Re^{i(kx+\tau t)}c_1(k,\tau)d\tau,$$
	the contribution to the left side of \eqref{fflem1} is bounded by
	\begin{equation}\label{fflem1l6}
		\bigg(\sum_{k\ne0}\int_\mathbb R\bigg|\frac{\widehat{T_2T_3}(k,\tau)}{\langle\tau+\alpha k^5-\beta k^3\rangle^\frac12}\bigg|^2d\tau\bigg)^\frac12.
	\end{equation}
Again $u_1$ and $u_2$ are thought to be local in time, and we shall use the estimate
	\begin{equation}\label{t2propcor}
	\|f\|_{X_\delta^{0,-\frac3{10}}}\lesssim_{\alpha,\beta}\|f\|_{L^\frac43(\mathbb T^2)};
	\end{equation}
indeed, for some function $h\in L^2$ with $\|h\|_{L^2}=1$, we see
\begin{align*}
	\|\eta(t/\delta)f\|_{X^{0,-\frac3{10}}}&=\|\chi_{[-2\delta,2\delta]}(t)\eta(t/\delta)f\|_{X^{0,-\frac3{10}}}\\
	&=\big\|\langle m+\alpha k^5-\beta k^3\rangle^{-\frac3{10}}\widehat{\eta(\cdot/\delta)f}\big\|_{l_{k,m}^2}\\
	&=\sum_{k,m}\langle m+\alpha k^5-\beta k^3\rangle^{-\frac3{10}}\widehat{\eta(\cdot/\delta)f}\,\widehat h\\
	&\le\|\eta(t/\delta)f\|_{L_{x,t}^\frac43}\|(\langle m+\alpha k^5-\beta k^3\rangle^{-\frac3{10}}\widehat h)^\vee\|_{L_{x,t}^4}\\
	&\lesssim_{\alpha,\beta}\|\eta(t/\delta)f\|_{L^\frac43}\|h\|_{L^2}.
\end{align*}
In the last inequality here, we used \eqref{t2propeq}.
Applying \eqref{t2propcor}, H\"older's inequality and \eqref{fflem1l4} to \eqref{fflem1l6}, we now have the estimate
\begin{align*}
\eqref{fflem1l6}\le\left\|\frac{\widehat{T_2T_3}}{\langle\tau+\alpha k^5-\beta k^3\rangle^\frac3{10}}\right\|_{l_k^2L_{\tau}^2}\lesssim_{\alpha,\beta}\|T_2T_3\|_{L^\frac43}&\le\|T_2\|_{L^4}\|T_3\|_{L^2}\\
	&\le\|u_1\|_{X^{s,\frac14}}\|u_2\|_{X^{s,\frac3{10}}}.
\end{align*}

The case \eqref{fflem1c3} is similar to \eqref{fflem1c2}. This proves \eqref{fflem1}.

\subsection*{Proof of \eqref{fflem2}}
Using \eqref{fflem1l2}, we now consider
\begin{align}\label{f}
&\frac{|k|^s|\mathcal F(\partial_x(u_1u_2))(k,\tau)|}{\langle\tau+\alpha k^5-\beta k^3\rangle}\\
\nonumber
&\lesssim_{\alpha,\beta,s}\sum_{k_1}\int_\mathbb R\!\frac{|k|c_1(k_1,\tau_1)c_2(k-k_1,\tau-\tau_1)}{\langle\tau\!\!\:+\!\!\:\alpha k^5\!\!\:-\!\!\:\beta k^3\rangle\langle\tau_1\!\!\:+\!\!\:\alpha k_1^5\!\!\:-\!\!\:\beta k_1^3\rangle^\frac14\langle\tau\!\!\:-\!\!\:\tau_1\!\!\:+\!\!\:\alpha(k\!\!\:-\!\!\:k_1)^5\!\!\:-\!\!\:\beta(k\!\!\:-\!\!\:k_1)^3\rangle^\frac14}d\tau_1.
\end{align}
We assume $u_1$ and $u_2$ to be local-in-time and divide the computation into the cases \eqref{fflem1c1}--\eqref{fflem1c3} again.

Consider the contribution of \eqref{fflem1c1}. We note that $\langle\tau+\alpha k^5-\beta k^3\rangle\sim k^4+|\tau+\alpha k^5-\beta k^3|$. 
To estimate the left side of \eqref{fflem2}, we consider a sequence $\{a_k\}$ such that $a_k\ge0$, $\sum a_k^2=1$.
Then, with $T_1$ and $T_2$ as in \eqref{tj}, the contribution is at most 
		\begin{equation}\label{fflem2l1}
\sum_k\int_\mathbb R\frac{a_k|k|}{k^4+|\tau+\alpha k^5-\beta k^3|}\widehat{T_1T_2}(k,\tau)d\tau.
\end{equation}
Let us define the function
$$T_4(x,t):=\sum_k\int_\mathbb Re^{i(kx+\tau t)}\frac{a_k|k|}{k^4+|\tau+\alpha k^5-\beta k^3|}d\tau$$
whose $L^2$-norm is
$$\|T_4\|_2\lesssim\bigg(\sum_k\int_\mathbb R\frac{a_k^2|k|^2}{(k^4+|\tau|)^2}d\tau\bigg)^\frac12\lesssim\bigg(\sum a_k^2\bigg)^\frac12=1.$$
By \eqref{fflem1l4}, we then get
		$$\eqref{fflem2l1}=\langle T_4,T_1T_2\rangle\le\|T_4\|_2\|T_1\|_4\|T_2\|_4\lesssim_{\alpha,\beta}\|u_1\|_{X^{s,\frac3{10}}}\|u_2\|_{X^{s,\frac3{10}}}.$$

For the second case \eqref{fflem1c2}, using \eqref{fflem1l2} we write
		\begin{align*}
			&\frac{|k|^s|\mathcal F(\partial_x(u_1u_2))(k,\tau)|}{\langle\tau+\alpha k^5-\beta k^3\rangle}\\
			&\lesssim_{\alpha,\beta,s}\sum_{k_1}\int_\mathbb R\frac{c_1(k_1,\tau_1)c_2(k-k_1,\tau-\tau_1)}{\langle\tau+\alpha k^5-\beta k^3\rangle\langle\tau-\tau_1+\alpha(k-k_1)^5-\beta(k-k_1)^3\rangle^\frac14}d\tau_1\\
			&=\frac1{\langle\tau\!\!\:+\!\!\:\alpha k^5\!\!\:-\!\!\:\beta k^3\rangle^{\frac12+}}\!\!\;\sum_{k_1}\!\!\;\int_\mathbb R\!\!\;\frac{c_1(k_1,\tau_1)c_2(k-k_1,\tau-\tau_1)}{\langle\tau\!\!\:+\!\!\:\alpha k^5\!\!\:-\!\!\:\beta k^3\rangle^{\frac12-}\langle\tau\!\!\:-\!\!\:\tau_1\!\!\:+\!\!\:\alpha(k\!\!\:-\!\!\:k_1)^5\!\!\:-\!\!\:\beta(k\!\!\:-\!\!\:k_1)^3\rangle^\frac14}d\tau_1.
		\end{align*}
To compute $l_k^2L_\tau^1$-norm, we first integrate in $\tau$ and use the Cauchy-Schwarz inequality:
		\begin{align*}
			&\int_\mathbb R\!\!\;\frac1{\langle\tau\!+\!\alpha k^5\!-\!\beta k^3\rangle^{\frac{1}{2}+}}\!\!\;\sum_{k_1}\!\!\;\int_\mathbb R\!\!\;\frac{c_1(k_1,\tau_1)c_2(k-k_1,\tau-\tau_1)}{\langle\tau\!+\!\alpha k^5\!-\!\beta k^3\rangle^{\frac{1}{2}-}\langle\tau\!-\!\tau_1\!+\!\alpha(k\!-\!k_1)^5\!-\!\beta(k\!-\!k_1)^3\rangle^{\frac14}}d\tau_1d\tau\\
			&\le\left(\int_\mathbb R\frac{1}{\langle\tau+\alpha k^5-\beta k^3\rangle^{1+}}d\tau\right)^\frac12\\
			&\quad\,\left(\int_\mathbb R\left(\sum_{k_1}\int_\mathbb R\frac{c_1(k_1,\tau_1)c_2(k-k_1,\tau-\tau_1)}{\langle\tau+\alpha k^5-\beta k^3\rangle^{\frac12-}\langle\tau-\tau_1+\alpha(k-k_1)^5-\beta(k-k_1)^3\rangle^\frac14}d\tau_1\right)^2\!d\tau\right)^\frac12.
		\end{align*}
What follows is similar to case \eqref{fflem1c2} in the proof of \eqref{fflem1}, except that the power $1/2$ of $\langle\tau+\alpha k^5-\beta k^3\rangle$ is replaced by $(1/2)-$. The last integral above is bounded by $\|u_1\|_{X^{s,\frac14}}\|u_2\|_{X^{s,\frac3{10}}}.$

The last case \eqref{fflem1c3} is similar to \eqref{fflem1c2}. This proves \eqref{fflem2}.

\subsubsection*{Proof of \eqref{t2propeq}}
It suffices to prove $\|f\|_{L^4(\mathbb T^2)}\le\|\chi_{[-\pi,\pi]}f\|_{X^{0,\frac3{10}}}$, from which we acquire for $\delta<1$
$$\|\chi_{[0,\delta]}(t)f\|_{L^4(\mathbb T\times\mathbb R)}
\le\|\eta(t/\delta)f\|_{L^4(\mathbb T^2)}\le\|\chi_{[-\pi,\pi]}\eta(t/\delta)f\|_{X^{0,\frac3{10}}}=\|\eta(t/\delta)f\|_{X^{0,\frac3{10}}}.$$
Note that $\|\cdot\|_{X^{s,b}}=\|e^{-Lt}\cdot\|_{H_x^sH_t^b}$. 
For $v:=e^{t(\alpha\partial_x^5+\beta\partial_x^3)}f$, we want to show $\|f\|_{L^4(\mathbb T^2)}\lesssim\|v\|_{L_x^2H_t^\frac3{10}}$.
Using the Fourier transform, one can write 
$$f(x,t)=\sum_ne^{i(nx-\alpha n^5t+\beta n^3t)}v_n(t)$$
and with $\Delta+n$ as a second variable
\begin{equation*}
(f\overline f)(x,t)=\sum_\Delta e^{i\Delta x}e^{-i(\alpha\Delta^5t-\beta\Delta^3t)}\sum_ne^{\lambda_{n,\Delta} in\Delta(n+\Delta)t}(v_n\overline{v_{n+\Delta}})(t),
\end{equation*}
where $\lambda_{n,\Delta}:=-5\alpha(n^2+\Delta^2+n\Delta)+3\beta$.
Then, by Plancherel's theorem, we have 
\begin{equation}\label{t2propffl2}
\|f\|_{L_{x,t}^4}^4=\|\mathcal F_x(f\overline f)\|_{l_\Delta^2L_t^2}^2=\bigg\|\sum_ne^{i\lambda_{n,\Delta} n\Delta(n+\Delta)t}v_n\overline{v_{n+\Delta}}\bigg\|_{l_\Delta^2L_t^2}^2. 
\end{equation}
First of all we may neglect negative $n$ in \eqref{t2propffl2} since
	\begin{align*}
		\bigg\|\sum_{n<0}e^{i\lambda_{n,\Delta} n\Delta(n+\Delta)t}v_n\overline{v_{n+\Delta}}\bigg\|_{l_\Delta^2}&=\bigg\|\sum_{n>0}\overline{e^{i\lambda_{-n,\Delta}(-n)\Delta(-n+\Delta)t}v_{-n}\overline{v_{-n+\Delta}}}\bigg\|_{l_\Delta^2}\\
		&=\bigg\|\sum_{n>0}e^{i\lambda_{n,-\Delta}n(-\Delta)(n-\Delta)t}\overline{v}_n\overline{\overline v_{n-\Delta}}\bigg\|_{l_\Delta^2}\\
		&=\bigg\|\sum_{n>0}e^{i\lambda_{n,\Delta} n\Delta(n+\Delta)t}\overline{v}_n\overline{\overline v_{n+\Delta}}\bigg\|_{l_\Delta^2}.
	\end{align*}
Also since
	\begin{align*}
		\bigg\|\sum_ne^{i\lambda_{n,\Delta} n\Delta(n+\Delta)t}v_n\overline{v_{n+\Delta}}\bigg\|_{l_{\Delta_{>0}}^2}&=\bigg\|\sum_ne^{i\lambda_{n-\Delta,\Delta}(n-\Delta)\Delta nt}v_{n-\Delta}\overline{v_n}\bigg\|_{l_{\Delta_{>0}}^2}\\
		&=\bigg\|\sum_ne^{-i\lambda_{n+\Delta,-\Delta}(n+\Delta)\Delta nt}v_{n+\Delta}\overline{v_n}\bigg\|_{l_{\Delta_{<0}}^2}\\
		&=\bigg\|\sum_ne^{i\lambda_{n,\Delta}n\Delta (n+\Delta)t}v_{n}\overline{v_{n+\Delta}}\bigg\|_{l_{\Delta_{<0}}^2},
	\end{align*}
we may neglect negative $\Delta$ in \eqref{t2propffl2} as well.	
	For $\Delta=0$,
	$$\sum_ne^{i\lambda_{n,\Delta}n\Delta(n+\Delta)t}v_n\overline{v_{n+\Delta}}=\sum_n|v_n|^2=\|v\|_{L_x^2}^2.$$
Hereafter unrestricted summation indices $n$ and $\Delta$ run over $\mathbb Z_{\ge0}$ and $\mathbb Z_{>0}$, respectively. 

Now, let us define for $j\ge0$,
	$$v_{n,j}(t):=\sum_{|m|\sim2^{j}}\widehat{v_n}(m) e^{imt}.$$
Applying the Littlewood-Paley decompositions $v_n=\sum_{j\ge0}v_{n,j}$, one has
\begin{equation}\label{t2propffl2e}
\eqref{t2propffl2}\lesssim\left\|\sum_j\sum_{k\le j}\bigg\|\sum_ne^{i\lambda_{n,\Delta}n\Delta(n+\Delta)t}v_{n,j}\overline{v_{n+\Delta,k}}\bigg\|_{L_t^2}\right\|_{l_\Delta^2}^2.
\end{equation}
We want to bound \eqref{t2propffl2e} by $C\|v\|_{L_x^2H_t^\frac3{10}}^4$. We now 
divide $\Delta$-summation into the three cases
	\begin{align}
		\label{t2propdec1}\Delta^5\le&2^j,\\
		\label{t2propdec2}\Delta^4\le&2^j<\Delta^5,\\
		\label{t2propdec3}&2^j<\Delta^4.
	\end{align}

We first consider the contribution of {\eqref{t2propdec1}}.
We evaluate $L_t^2$-norm of \eqref{t2propffl2e} as follows:

\begin{align}
\nonumber&\bigg\|\sum_ne^{i\lambda_{n,\Delta}n\Delta(n+\Delta)t} v_{n,j}\overline{v_{n+\Delta,k}}\bigg\|_{L_t^2}\\
\nonumber&\qquad\le\sum_n\|v_{n,j}\overline{v_{n+\Delta,k}}\|_{L_t^2}=\sum_n\bigg\|\sum_{m_1}\widehat{v_{n,j}}(m_1)\widehat{\overline{v_{n+\Delta,k}}}(\cdot-m_1)\bigg\|_{l_m^2}\\
\nonumber&\qquad\le\sum_n\bigg\|\bigg(\sum_{m_1\atop\widehat{\overline{v_{n+\Delta,k}}}(\cdot-m_1)\ne0}1\bigg)^\frac12\bigg(\sum_{m_1}\big|\widehat{v_{n,j}}(m_1)\widehat{\overline{v_{n+\Delta,k}}}(\cdot-m_1)\big|^2\bigg)^\frac12\bigg\|_{l_m^2}\\
\label{t2propdec1bd0}&\qquad\lesssim2^{k/2}\sum_n\big\|v_{n,j}\big\|_{L_t^2}\big\|v_{n+\Delta,k}\big\|_{L_t^2}.
\end{align}
Here we used the Cauchy-Schwarz inequality and the fact that $\widehat{\overline{v_{n+\Delta,k}}}$ is non-zero for at most a constant multiple of $2^k$ values as well as $k\le j$.
Using the Cauchy-Schwarz inequality again, \eqref{t2propdec1bd0} can be further estimated as
\begin{align}
\nonumber&\bigg\|\sum_ne^{i\lambda_{n,\Delta}n\Delta(n+\Delta)t} v_{n,j}\overline{v_{n+\Delta,k}}\bigg\|_{L_t^2}\\
\nonumber&\ \ \lesssim2^{\frac{k}{2}}\bigg(\sum_n\|v_{n,j}\|_{L_t^2}^2\bigg)^{\frac{1}{2}}\bigg(\sum_n\|v_{n+\Delta,k}\|_{L_t^2}^2\bigg)^{\frac{1}{2}}\\
\label{t2propdec1bd1}&\ \ \le2^{\frac{k}{5}}2^{-\frac{3}{10}j}\bigg(\sum_n\sum_{|m|\sim2^j}\langle m\rangle^{\frac{3}{5}}\big|\widehat{v_n}(m)\big|^2\bigg)^{\frac{1}{2}}\bigg(\sum_n\sum_{|m|\sim2^k}\langle m\rangle^{\frac{3}{5}}\big|\widehat{v_{n+\Delta}}(m)\big|^2\bigg)^{\frac{1}{2}}.
\end{align}
Note that the last factor is at most $\|v\|_{L_x^2H_t^\frac3{10}}$. 
Applying  \eqref{t2propdec1bd1} to \eqref{t2propffl2e}, we have
\begin{align}
\nonumber\bigg\|\sum_{2^j \ge \Delta^5}&\sum_{k\le j}\bigg\|\sum_ne^{i\lambda_{n,\Delta}n\Delta(n+\Delta)t}v_{n,j}\overline{v_{n+\Delta,k}}\bigg\|_{L_t^2}\bigg\|_{l_\Delta^2}^2\\
\label{t2propdec1bd2}&\lesssim\bigg\|\sum_{2^j \ge \Delta^5}2^{-\frac j{10}}\bigg(\sum_n\sum_{|m|\sim2^j}\langle m\rangle^\frac35\big|\widehat{v_n}(m)\big|^2\bigg)^\frac12 \bigg\|_{l_\Delta^2}^2\|v\|^2_{L_x^2H_t^\frac3{10}}.
\end{align}
To estimate the first factor of \eqref{t2propdec1bd2}, we write $\sum_\Delta=\sum_{l\ge0}\sum_{\Delta\sim2^l}$ and set $j=5l+j'$. Then using the Cauchy-Schwarz inequality to the $j'$-sum, we have 
\begin{align*}
&\sum_{l\ge0}\sum_{\Delta\sim2^l}2^{-l}\bigg(\sum_{2^{j'}\gtrsim1}2^{-\frac{j'}{20}}\bigg(2^{-\frac{j'}{10}}\sum_n\sum_{|m|\sim2^{5l+j'}}\langle m\rangle^\frac35\big|\widehat{v_n}(m)\big|^2\bigg)^\frac12\bigg)^2\\
&\quad\lesssim\sum_{l\ge0}\bigg(\Big(\sum_{2^{j'}\gtrsim1}2^{-\frac{j'}{10}}\Big)^\frac12\Big(\sum_{2^{j'}\gtrsim1}2^{-\frac{j'}{10}}\sum_n\sum_{|m|\sim2^{5l+j'}}\langle m\rangle^\frac35\big|\widehat{v_n}(m)\big|^2\Big)^\frac12\bigg)^2\lesssim \|v\|_{L_x^2H_t^\frac3{10}}^2
\end{align*}
and the contribution of \eqref{t2propdec1} to \eqref{t2propffl2e} is at most a constant multiple of $\|v\|_{L_x^2H_t^\frac3{10}}^4$.

Now we consider the second case \eqref{t2propdec2}. 
Since 
$$\mathcal F(v_{n,j}\overline{v_{n+\Delta,k}})(m)=\sum_{m_1}\widehat{v_{n,j}}(m-m_1)\widehat{\overline{v_{n+\Delta,k}}}(m_1)$$
and the summand is non-zero only when $|m-m_1|\sim2^j$ and $|m_1|\sim2^k$, we see that for $k\le j$,
$$\text{supp}(\mathcal F(v_{n,j}\overline{v_{n+\Delta,k}}))\subset[\sim-2^j,\sim2^j].$$
Recall that $n\ge0$, $\Delta>0$ and we set $\Gamma_\Delta(n):=\lambda_{n,\Delta} n\Delta(n+\Delta)$.
For fixed $\Delta$, the gaps of the sequence $\{\Gamma_\Delta(n)\}_n$ are at least $C_{\alpha,\beta}\Delta^4$.
We shall split $\sum_n$ into $\sum_{0\le a\le d-1}\sum_{n\in\mathcal{N}_a}$ with $d\sim_{\alpha,\beta}2^j/\Delta^4$; $\mathcal{N}_a$ denotes the equivalence class of $a$ mod $d$. Note that for $n_1,n_2\in\mathcal N_a$, $n_1\ne n_2$, we have $|\Gamma_{\Delta}(n_1)-\Gamma_{\Delta}(n_2)|\gg2^j$.
By Parseval's identity and orthogonality considerations, we have 
\begin{align*}
&\bigg\|\sum_{n\in\mathcal N_a}e^{i\lambda_{n,\Delta}n\Delta(n+\Delta)t}v_{n,j}\overline{v_{n+\Delta,k}}\bigg\|_{L_t^2}^2\\
&=\sum_{n\in\mathcal N_a} \|v_{n,j}\overline{v_{n+\Delta,k}}\|_{L_t^2}^2\\
&\quad+\!\int\!\!\!\sum_{n_1,n_2\in\mathcal N_a\atop n_1\ne n_2}\!\!\!e^{i\lambda_{n_1,\Delta}n_1\Delta(n_1+\Delta)t}(v_{n_1,j}\overline{v_{n_1+\Delta,k}})(t)\,\overline{e^{i\lambda_{n_2,\Delta}n_2\Delta(n_2+\Delta)t}(v_{n_2,j}\overline{v_{n_2+\Delta,k}})(t)}\,dt\\
&=\sum_{n\in\mathcal N_a}\|v_{n,j}\overline{v_{n+\Delta,k}}\|_{L_t^2}^2\\
&\quad+\sum_m\sum_{n_1,n_2\in\mathcal N_a\atop n_1\ne n_2}\mathcal F(v_{n_1,j}\overline{v_{n_1+\Delta,k}})(m-\Gamma_{\Delta}(n_1))\overline{\mathcal F(v_{n_2,j}\overline{v_{n_2+\Delta,k}})}(m-\Gamma_{\Delta}(n_2))\\
&=\sum_{n\in\mathcal N_a}\|v_{n,j}\overline{v_{n+\Delta,k}}\|_{L_t^2}^2.
\end{align*}
Thus, getting back to \eqref{t2propffl2e}, we first see
\begin{align*}
\bigg\|\sum_{a=0}^{d-1}\sum_{n\in\mathcal N_a}e^{i\lambda_{n,\Delta}n\Delta(n+\Delta)t}v_{n,j}\overline{v_{n+\Delta,k}}\bigg\|_{L_t^2}
&\le\sum_{a=0}^{d-1}\bigg\|\sum_{n\in\mathcal N_a}e^{i\lambda_{n,\Delta}n\Delta(n+\Delta)t}v_{n,j}\overline{v_{n+\Delta,k}}\bigg\|_{L_t^2}\\
&=\sum_{a=0}^{d-1}\bigg(\sum_{n\in\mathcal N_a}\|v_{n,j}\overline{v_{n+\Delta,k}}\|_{L_t^2}^2\bigg)^{\frac12}\\
&\lesssim_{\alpha,\beta}\bigg(\frac{2^j}{\Delta^4}\bigg)^{\frac12}\bigg(\sum_n\|v_{n,j}\overline{v_{n+\Delta,k}}\|_{L_t^2}^2\bigg)^{\frac12}\\
&\lesssim\frac{2^{\frac j2}}{\Delta^2}2^{\frac{k}{2}}\bigg(\sum_n\|v_{n,j}\|_{L_t^2}^2\|v_{n+\Delta,k}\|_{L_t^2}^2\bigg)^{\frac12}.
\end{align*}
The last inequality follows from similar steps to \eqref{t2propdec1bd0}. Then applying the Cauchy-Schwarz inequality, the contribution of \eqref{t2propdec2} to \eqref{t2propffl2e} is bounded by
\begin{align*}
&\sum_\Delta\frac1{\Delta^4}\bigg(\sum_{j,k\atop{j\ge k\atop\Delta^4\le2^j<\Delta^5}}2^{\frac j5}2^{\frac k5}\bigg(\sum_n\Big(\sum_{|m|\sim2^j}\langle m\rangle^{\frac35}\big|\widehat{v_n}(m)\big|^2\sum_{|m|\sim2^k}\langle m\rangle^{\frac35}\big|\widehat{v_{n+\Delta}}(m)|^2\Big)\bigg)^{\frac12}\bigg)^2\\
&\quad\le\sum_\Delta\frac1{\Delta^4}\bigg(\sum_{j,k\atop{j\ge k\atop\Delta^4\le2^j<\Delta^5}}2^{\frac {2j}5}2^{\frac {2k}5}\bigg)\bigg(\sum_{j,k\atop{j\ge k\atop\Delta^4\le2^j<\Delta^5}}\sum_n\|v_{n,j}\|_{L_x^2H_t^\frac3{10}}^2\|v_{n+\Delta,k}\|_{L_x^2H_t^\frac3{10}}^2\bigg)\\
&\quad\lesssim\sum_n\|v_{n}\|_{L_x^2H_t^\frac3{10}}^2\sum_\Delta\|v_{n+\Delta}\|_{L_x^2H_t^\frac3{10}}^2=\|v\|_{L_x^2H_t^\frac3{10}}^4.
\end{align*}
Finally, the last case \eqref{t2propdec3} can be handled in a similar manner to the case \eqref{t2propdec2} except that one does not have to split $\sum_n$; we omit the details.\qed

\begin{rem}\label{remref}
Here \eqref{w3}, \eqref{t2propeq}, \eqref{t2propcor} are sharper than their counterparts in \cite B because the Kawahara equation has better dispersion. In particular the increment of the frequency sequence in the above proof, $\{\lambda_{n,\Delta}n\Delta(n+\Delta)\}_n$, is greater, and we also have a higher bound in \eqref{powerbound}. In fact, we have $\|f\|_{L^4}\lesssim\|f\|_{X^{0,\frac14\left(1+\frac1{1+\kappa}\right)}}$ for the gaps of $\Delta^\kappa$, and in such case, for \eqref{w3} we have the following:
\begin{align*}
&\|\partial_x(u_1u_2)\|_{X_\delta^{s,\max\left\{\frac1\kappa,\frac14\left(1+\frac1{1+\kappa}\right)\right\}}}\\
&\lesssim\!\|u_1\|_{X_\delta^{s,\max\left\{\frac1\kappa,\frac14\left(1+\frac1{1+\kappa}\right)\right\}}}\|u_2\|_{X_\delta^{s,\frac14\left(1+\frac1{1+\kappa}\right)}}\!+\!\|u_1\|_{X_\delta^{s,\frac14\left(1+\frac1{1+\kappa}\right)}}\|u_2\|_{X_\delta^{s,\max\left\{\frac1\kappa,\frac14\left(1+\frac1{1+\kappa}\right)\right\}}},\\
&\left\|\frac{\langle k\rangle^s\mathcal F{(\partial_x(u_1u_2))}}{\langle\tau+\alpha k^5-\beta k^3\rangle^{\max\left\{\frac1{2\kappa}+\frac12,\frac14\left(1+\frac1{1+\kappa}\right)+\frac12+\right\}}}\right\|_{l_k^2L_\tau^1}\\
&\lesssim\!\|u_1\|_{X_\delta^{s,\max\left\{\frac1\kappa,\frac14\left(1+\frac1{1+\kappa}\right)\right\}}}\|u_2\|_{X_\delta^{s,\frac14\left(1+\frac1{1+\kappa}\right)}}\!+\!\|u_1\|_{X_\delta^{s,\frac14\left(1+\frac1{1+\kappa}\right)}}\|u_2\|_{X_\delta^{s,\max\left\{\frac1\kappa,\frac14\left(1+\frac1{1+\kappa}\right)\right\}}}.
\end{align*}
\end{rem}
\begin{rem}
Although not necessary in our calculation, one can in fact have \eqref{w3} up to $s\ge-1/2$. We see
\begin{equation}\label{amidone}
\frac{|k|^{s+1}}{|k_1|^s|k-k_1|^s}=\frac{|k|^{s+\frac12}|k|^\frac12}{|k_1(k-k_1)|^s}\lesssim\frac{|k_1(k-k_1)|^{s+\frac12}|k|^\frac12}{|k_1(k-k_1)|^s}=|k_1(k-k_1)k|^\frac12.
\end{equation}
Note that one may use a sharper lower bound for \eqref{powerbound}:
\begin{align*}
|(\tau+\alpha k^5-\beta k^3)-(\tau_1+\alpha k_1^5-\beta k_1^3)&-(\tau-\tau_1+\alpha(k-k_1)^5-\beta(k-k_1)^3)|\\
&\!\!\!\!\!\!=\frac12|k_1(k-k_1)k(5\alpha(k^2+k_1^2+(k-k_1)^2)-6\beta)|\\
&\!\!\!\!\!\!\gtrsim_{\alpha,\beta}|k_1(k-k_1)k|(k^2+k_1^2+(k-k_1)^2)\\
&\!\!\!\!\!\!\ge|k_1(k-k_1)k|^\frac53.
\end{align*}
This is the rightmost side of \eqref{amidone} to the power of 10/3, from which we may replace $|k|^{s+1}|k_1|^{-s}|k-k_1|^{-s}$ in \eqref{fflem1l1} and \eqref f with $|k_1(k-k_1)k|^\frac12$ instead. This leads to \eqref{w3}.
At the cost of a higher $b$-index in $X^{s,b}$-norm on the right-hand side of \eqref{w3}, one may go further: for $s\ge-2/3$, we have
$$\|\partial_x(u_1u_2)\|_{Z_\delta^s}\lesssim\|u_1\|_{X_\delta^{s,\frac13}}\|u_2\|_{X_\delta^{s,\frac3{10}}}+\|u_1\|_{X_\delta^{s,\frac3{10}}}\|u_2\|_{X_\delta^{s,\frac13}}\le\|u_1\|_{X_\delta^{s,\frac13}}\|u_2\|_{X_\delta^{s,\frac13}}.$$
One may indeed extend this remark similarly to Remark \ref{remref}. See \cite{KPV} for an analogous result for the KdV equation.
\end{rem}


\end{document}